\documentclass{article}

\usepackage{epsfig,epsf,fancybox}
\usepackage{amsmath}
\usepackage{mathrsfs}
\usepackage{amssymb}
\usepackage{graphicx}
\usepackage{color}
\usepackage{multirow}
\usepackage{paralist}
\usepackage{verbatim}
\usepackage{galois}
\usepackage{algorithm}
\usepackage{algorithmic}
\usepackage{boxedminipage}
\usepackage{booktabs}
\usepackage{accents}
\usepackage{stmaryrd}
\usepackage{subfig}
\usepackage{appendix}
\usepackage{epstopdf}
\usepackage{appendix}

\usepackage{amsmath}
\usepackage{amsfonts}
\usepackage{amssymb}
\usepackage{amscd}
\usepackage{amsbsy}
\usepackage{amsthm}
\usepackage{graphicx}
\usepackage{psfrag}
\usepackage{color}

\usepackage[twoside]{geometry}
\newfont{\myfnt}{cmssi10 scaled 1440}
\numberwithin{equation}{section}
\catcode`@=11
\def\ps@nk{\def\@oddhead{\vbox{\hbox to \hsize{\pic \footnotesize \it \shorttitle
				\hfill \rm \thepage} \vspace{1mm} \vspace*{-2mm}}}
	\def\@evenhead{\vbox{\hbox to \hsize{\pic \footnotesize \rm \thepage \hfill \it \shortauthor}
			\vspace{1mm} \vspace*{-2mm}}}
	\def\@oddfoot{} \def\@evenfoot{}}
\def\ps@first{\def\@oddhead{}
	\def\@evenhead{}
	\def\@oddfoot{} \def\@evenfoot{}}
\def\ps@total{\def\@oddhead{\vbox{\hbox to \hsize{\footnotesize \rm \hfill TOTAL\ \ CONTENTS
				\hfill \thepage} \vspace{1mm} \hrule \vspace*{-2mm}}}
	\def\@evenhead{\vbox{\hbox to \hsize{\footnotesize \rm \thepage \hfill CHIN.\ \ ANN.\ \ MATH.
				\hfill} \vspace{1mm} \hrule \vspace*{-2mm}}}
	\def\@oddfoot{} \def\@evenfoot{}}
\newtheoremstyle{thmstyle}% name
{6pt}%      Space above
{6pt}%      Space below
{\it}%         Body font
{}%         Indent amount (empty = no indent, \parindent = para indent)
{\bf}% Thm head font
{}%        Punctuation after thm head
{.5em}%     Space after thm head: " " = normal interword space;
{}%         Thm head spec (can be left empty, meaning `normal')
\newtheoremstyle{remstyle}% name
{6pt}%      Space above
{6pt}%      Space below
{\rm}%         Body font
{}%         Indent amount (empty = no indent, \parindent = para indent)
{\bf}% Thm head font
{}%        Punctuation after thm head
{.5em}%     Space after thm head: " " = normal interword space;
{}%         Thm head spec (can be left empty, meaning `normal')

\def\Sec{\@Startsection{section}{1}{\z@}
	{-3.5ex \@plus -1ex \@minus -.2ex}%
	{2.3ex \@plus.2ex}%
	{\normalfont\large\bfseries\boldmath}}
\def\@Startsection#1#2#3#4#5#6{%
	\if@noskipsec \leavevmode \fi
	\par
	\@tempskipa #4\relax
	\@afterindenttrue
	\ifdim \@tempskipa <\z@
	\@tempskipa -\@tempskipa \@afterindentfalse
	\fi
	\if@nobreak
	\everypar{}%
	\else
	\addpenalty\@secpenalty\addvspace\@tempskipa
	\fi
	\@ifstar
	{\@ssect{#3}{#4}{#5}{#6}}%
	{\@dblarg{\@Sect{#1}{#2}{#3}{#4}{#5}{#6}}}}
\def\@Sect#1#2#3#4#5#6[#7]#8{%
	\ifnum #2>\c@secnumdepth
	\let\@svsec\@empty
	\else
	\refstepcounter{#1}%
	\protected@edef\@svsec{\@seccntformat{#1}\relax}%
	\fi
	\@tempskipa #5\relax
	\ifdim \@tempskipa>\z@
	\begingroup
	#6{%
		\@hangfrom{\hskip #3\relax\@svsec \hskip -2.5mm}%
		\interlinepenalty \@M #8\@@par}
	\endgroup
	\csname #1mark\endcsname{#7}%
	\addcontentsline{toc}{#1}{%
		\ifnum #2>\c@secnumdepth \else
		\protect\numberline{\csname the#1\endcsname}%
		\fi
		#7}%
	\else
	\def\@svsechd{%
		#6{\hskip #3\relax
			\@svsec #8}%
		\csname #1mark\endcsname{#7}%
		\addcontentsline{toc}{#1}{%
			\ifnum #2>\c@secnumdepth \else
			\protect\numberline{\csname the#1\endcsname}%
			\fi
			#7}}%
	\fi
	\@xsect{#5}}
\renewenvironment{abstract}{%
	\small
	\quotation
	\noindent {\bfseries \abstractname } }%
{\if@twocolumn\else\endquotation\fi}

\def\Subsec{\@StartSubsection{subsection}{2}{\z@}%
	{-3.25ex\@plus -1ex \@minus -.2ex}%
	{1.5ex \@plus .2ex}%
	{\normalfont\normalsize\bfseries\boldmath}}
\def\@StartSubsection#1#2#3#4#5#6{%
	\if@noskipsec \leavevmode \fi
	\par
	\@tempskipa #4\relax
	\@afterindenttrue
	\ifdim \@tempskipa <\z@
	\@tempskipa -\@tempskipa \@afterindentfalse
	\fi
	\if@nobreak
	\everypar{}%
	\else
	\addpenalty\@secpenalty\addvspace\@tempskipa
	\fi
	\@ifstar
	{\@ssect{#3}{#4}{#5}{#6}}%
	{\@dblarg{\@SubSect{#1}{#2}{#3}{#4}{#5}{#6}}}}
\def\@SubSect#1#2#3#4#5#6[#7]#8{%
	\ifnum #2>\c@secnumdepth
	\let\@svsec\@empty
	\else
	\refstepcounter{#1}%
	\protected@edef\@svsec{\@seccntformat{#1}\relax}%
	\fi
	\@tempskipa #5\relax
	\ifdim \@tempskipa>\z@
	\begingroup
	#6{%
		\@hangfrom{\hskip #3\relax\@svsec\hskip -1.5mm}%
		\interlinepenalty \@M #8\@@par}
	\endgroup
	\csname #1mark\endcsname{#7}%
	\addcontentsline{toc}{#1}{%
		\ifnum #2>\c@secnumdepth \else
		\protect\numberline{\csname the#1\endcsname}%
		\fi
		#7}%
	\else
	\def\@svsechd{%
		#6{\hskip #3\relax
			\@svsec #8}%
		\csname #1mark\endcsname{#7}%
		\addcontentsline{toc}{#1}{%
			\ifnum #2>\c@secnumdepth \else
			\protect\numberline{\csname the#1\endcsname}%
			\fi
			#7}}%
	\fi
	\@xsect{#5}}
\def\list#1#2{\ifnum \@listdepth >5\relax \@toodeep \else \global
	\advance \@listdepth\@ne \fi \rightmargin \z@ \listparindent\z@
	\itemindent\z@ \csname @list\romannumeral\the\@listdepth\endcsname
	\def\@itemlabel{#1}\let\makelabel\@mklab \@nmbrlistfalse #2\relax
	\@trivlist \parskip 0pt \parindent\listparindent \advance \linewidth
	-\rightmargin \advance\linewidth -\leftmargin \advance\@totalleftmargin
	\leftmargin \parshape \@ne \@totalleftmargin \linewidth \ignorespaces}
\renewcommand{\@makecaption}[2]{\begin{center}#1. #2\end{center}}\catcode`@=12 

\theoremstyle{thmstyle}
\newtheorem{theorem}{\indent Theorem}[section]
\newtheorem{lemma}{\indent Lemma}[section]

\newtheorem{definition}{\indent Definition}[section]

\newtheorem{problem}{\indent Problem}[section]
\theoremstyle{remstyle}

\def\endproof{\hfill $\Box$ \vskip 0.4cm}

\newsavebox{\mygraphic}

\def\pic{\begin{picture}(0,0) \put(-210,-1250){\usebox{\mygraphic}} \end{picture}}
\newfont{\HUGEbf}{cmbx10 scaled 3500}
\definecolor{gray}{rgb}{0.9,0.9,0.9}
\def\thebibliography#1{\section*{\bf \large References}
	\list{[\arabic{enumi}]} {\settowidth \labelwidth{[#1]} \leftmargin
		\labelwidth \advance \leftmargin \labelsep \usecounter{enumi}}
	\def\newblock{\hskip .11em plus .33em minus .07em} \footnotesize \sloppy \clubpenalty
	4000 \widowpenalty 4000 \sfcode`\.=1000 \relax}

\newcommand{\ee}{\tilde{\mathbb{E}}_\tau}

\newcommand{\f}{\mathscr{F}}
\newcommand{\lr}{\mathcal{L}}
\newcommand{\e}{\mathbb{E}}
\newcommand{\br}{\mathbb{R}}
\newcommand{\pr}{\mathcal{P}}
\newcommand{\s}{\mathbb{S}}
\newcommand{\ii}{\mathcal{I}}
\newcommand{\iii}{\mathbb{I}}
\newcommand{\m}{\mathcal{M}}

\newcommand{\argmin}{\mathop{\rm argmin}}

\newcommand{\sr}{\mathcal{S}}

\newcommand{\ba}{\begin{array}}
	\newcommand{\ea}{\end{array}}

\def\firstpage{1}

\def\shorttitle{Mean Field Games with Common Noises and FBSDEs} 
\def\shortauthor{{\it Z. Y. Huang\ \ and\ \ S. J. Tang}} % ×÷Õß

\title{\Large \bf \boldmath\ \\ Mean Field Games with Common Noises and Conditional Distribution Dependent FBSDEs$^{\ast}$} % under a Weak Monotonicity Condition

\author{\large  Ziyu HUANG$^1$\qquad Shanjian TANG$^{2}$} 

\date{}

\begin{document}
	\maketitle
	\thispagestyle{first}
	\renewcommand{\thefootnote}{\fnsymbol{footnote}}
	\footnotetext{\hspace*{-5mm} \begin{tabular}{@{}r@{}p{13.4cm}@{}}
			%& Manuscript received  \\ %date
			$^1$ & School of Mathematical Sciences, Fudan University,
			Shanghai 200433, China.\\
			&{E-mail: 19110180044@fudan.edu.cn} \\
			$^{2}$ & Department of Finance and Control Sciences, School of Mathematical Sciences, Fudan University,
			Shanghai 200433, China. {E-mail: sjtang@fudan.edu.cn} \\
			$^{\ast}$ & This work was partially supported by National Key R\&D Program of China under grant 2018YFA0703900 and by National Natural Science Foundation of China under grants 11631004 and 12031009. 
	\end{tabular}}
	\renewcommand{\thefootnote}{\arabic{footnote}}
	
	\begin{abstract}
		In this paper, we consider the mean field game with a common noise and allow the state coefficients to vary with the conditional distribution in a nonlinear way. We assume that the cost function satisfies a convexity and a weak monotonicity property. We use the sufficient Pontryagin principle for optimality to transform the mean field control problem into existence and uniqueness of solution of conditional distribution dependent forward-backward stochastic differential equation (FBSDE). We prove the existence and uniqueness of solution of the conditional distribution dependent FBSDE when the dependence of the state on the conditional distribution is sufficiently small, or when the convexity parameter of the running cost on the control is sufficiently large. Two different methods are developed. The first method is based on a continuation of the coefficients, which is developed for FBSDE by Hu and Peng \cite{YH2}. We apply the method to conditional distribution dependent FBSDE. The second method is to show the existence result on a small time interval by Banach fixed point theorem and then extend the local solution to the whole time interval.
		
		\vskip 4.5mm
		
		\noindent \begin{tabular}{@{}l@{ }p{10.1cm}} {\bf Keywords } &
			mean field games, common noises, FBSDEs, stochastic maximum principle
		\end{tabular}
		
		\noindent {\bf 2000 MR Subject Classification } 93E20, 60H30, 60H10, 49N70, 49J99
	\end{abstract}
	
	\baselineskip 14pt
	
	\setlength{\parindent}{1.5em}
	
	\setcounter{section}{0}
	
	\section{Introduction}
	Mean field games (MFGs) were proposed by Lasry and Lions in a serie of papers \cite{JM1,JM2,JM3} and also independently by Huang, Caines and Malhamé \cite{MH}, under the different name of Nash Certainty Equivalence. They are sometimes approached by symmetric, non-cooperative stochastic differential games of interacting $N$ players. To be specific, each player solves a stochastic control problem with the cost and the state dynamics depending not only on his own state and control but also on other players' states. The interaction among the players can be weak in the sense that one player is influenced by the other players only through the empirical distribution. In view of the theory of McKean–Vlasov limits and propagation of chaos for uncontrolled weakly interacting particle systems \cite{AS}, it is expected to have a convergence for $N$-player game Nash equilibria by assuming independence of the random noise in the players' state processes and some symmetry conditions of the players. The literature in this area is huge. See \cite{PC} for a summary of a series of Lions' lectures given at the Collége de France. Carmona and Delarue approached the MFG problem from a probabilistic point of view. See \cite{CR,PA,CR3}. There are rigorous results about construction of $\epsilon$-Nash equilibria for $N$-player games. See for example \cite{DL,MH1,CR4,PA,VN}.
	
	In most studies mentioned above, the noises in each player's state dynamic are assumed to be independent and the empirical distribution of players' states is deterministic in the limit. See \cite{CR5} on a model of inter-bank borrowing and lending, where noises of players are dependent.
	
	The presence of a common noise clearly adds extra complexity to the problem as the empirical distribution of players' state becomes stochastical in the limit. Following a PDE approach, Pham and Wei \cite{HP3} studied the dynamic programming for optimal control of stochastic McKean-Vlasov dynamics; in particular, Pham \cite{HP1} solved the optimal control problem for a linear conditional McKean-Vlasov equation with a quadratic cost functional. Carmona and Delarue \cite{PA} consider the mean field game without common noises. They use a probabilistic approach based on the stochastic maximum principle (SMP) within a linear-convex framework. Nonetheless, their arguments of using Schauder fixed-point theorem to a compact subset of deterministic flows of probability measures, is difficult to be adapted to the case of common noises. Yu and Tang \cite{JY} considered mean field games with degenerate state- and distribution-dependent noise. Ahuja \cite{SA} studied a  simple linear model of the mean field games in the presence of common noise with the terminal cost be convex and weakly monotone. The statistics of the state process occurs in the McKean-Vlasov forward-backward stochastic differential equation (FBSDE) arising from the stochastic maximum principle as the distribution conditioned on the common noise. Ahuja et al.  \cite{SAWR} further consider a system of FBSDEs with monotone functionals and then solve the mean field game with a common noise within a linear-convex setting for weakly monotone cost functions. However, their state dynamics do not depend on the statistics of the state. The monotone condition usually fails to hold for the conditional distribution dependent FBSDE  if the state dynamic depends on the conditional distribution of the state. 
	
	In this paper, we consider the mean field game with a common noise and allow the state coefficients to vary with the conditional distribution in a nonlinear way. We use the sufficient Pontryagin principle for optimality to transform the mean field control problem into existence and uniqueness of solution of conditional distribution dependent FBSDE. We prove the existence and uniqueness of solution of the conditional distribution dependent FBSDE when the dependence of the state coefficient on the conditional distribution is sufficiently small, or when the convexity parameter of the running cost on the control is sufficiently large. To accomplish this, we assume that the terminal cost and the running cost are convex and weakly monotone. We develop two different methods to show the existence and uniqueness result.
	
	The first method is based on a continuation of the coefficients, which is developed for FBSDE by Hu and Peng \cite{YH2}. With this method, Carmona and Delarue \cite{CR} solve a linear case without common noises % where the state distribution appears in the form of the expectation, 
	and Ahuja et al. \cite{SAWR} solve that mean field games with common noises within a linear-convex setting when the state dynamic is independent of the conditional distribution of state. %In our paper, the method is adapted to the conditional McKean-Vlasov type FBSDE and no longer require the state distribution appears as the expected state. 
	
	The second method, inspired by \cite{SA}, is first to show the existence result on a small  time interval by Banach fixed point theorem and then to extend the local solution to the whole time. Ahuja \cite{SA} showed the existence and uniqueness result for the particular MFG with common noises for the linear state $dX_t=\alpha_tdt+\sigma dW_t+\tilde{\sigma}d\tilde{W}_t$, where $\alpha_t$ is the control, $(W,\tilde{W})$ is a two-dimensional standard Brownian motion and $(\sigma,\tilde{\sigma})$ is constant. We shall consider a more general model. All the coefficients of our state equation are allowed to depend on the control, the state and the conditional distribution of state. More assumptions in the second method are required to derive the existence result, while the probabilistic properties as well as the sensitivity of the FBSDEs have their own interests.
	
	The paper is organized as follows. In Section~\ref{problem_formulation}, we introduce our model and formulate the main problem. In Section~\ref{SMP}, we use the sufficient Pontryagin principle for optimality to transform the control problem into an existence and uniqueness problem of a conditional distribution dependent FBSDE. The existence and uniqueness result of the conditional distribution dependent FBSDE is stated and proved with different methods in Sections~\ref{M1} and \ref{M2}. Appendix contains the proofs of main lemmas given in Sections~\ref{M1} and \ref{M2}.
	
	\section{Problem Formulation}\label{problem_formulation}
	In this section, we describe our stochastic differential game model, and then formulate the limit problem of the $N$-player game as a MFG with a common noise.
	
	\subsection{Notations}
	Let $(\Omega,\f,\{\f_t,0\le t\le T\},\mathbb{P})$ denote a complete filtered probability space augmented by all the $\mathbb{P}$-null sets on which a one-dimensional Brownian motion $\{W_t,{0\le t\le T}\}$ is defined.  Let $\tilde{\f}=\{\tilde{\f}_t,0\le t\le T\}$ be a subfiltration of $\f$. $\lr(\cdot|\tilde{\f}_t)$ is the law conditioned at $\tilde{\mathscr{F}}_t$ for $t\in[0,T]$.
	
	Let $\lr^2_{\f_t}$ denote the set of all $\f_t$-measurable square-integrable $\br$-valued random variables. Let $\lr^2_{\f}(0,T)$ denote the set of all $\f_t$-progressively-measurable $\br$-valued processes $\alpha=(\alpha_t)_{0\le t\le T}$ such that 
	\begin{equation*}
		\e\big[\int_0^T |\alpha_t|^2dt\big]<+\infty.
	\end{equation*}
	Let $\mathcal{S}^2_{\f}(0,T)$ denote the set of all $\f_t$-progressively-measurable $\br$-valued processes $\beta=(\beta_t)_{0\le t\le T}$ such that 
	\begin{equation*}
		\e[\sup_{0\le t\le T} |\beta_t|^2dt]<+\infty.
	\end{equation*}
	We define similarly the spaces $\lr_{\f}(s,t)$ and $\mathcal{S}^2_{\f}(s,t)$ for any $0\le s<t\le T$.
	
	Let $\pr(\br)$ denote the space of all Borel probability measures on $\br$, and $\pr_2(\br)$ the space of all probability measures $m\in\pr(\br)$ such that 
	\begin{equation*}
		\int x^2dm(x)<\infty.
	\end{equation*}
	The Wasserstein distance is defined on $\pr_2(\br)$ by
	\begin{equation*}
		W_2(m_1,m_2)=\big(\inf_{\gamma\in\Gamma(m_1,m_2)}\int_{\br^2}|x(\omega_1)-x(\omega_2)|^2 d\gamma(\omega_1,\omega_2)\big)^{\frac{1}{2}},\quad m_1,m_2\in\pr_2(\br),
	\end{equation*}
	where $\Gamma(m_1,m_2)$ denotes the collection of all probability measures on $\br^2$ with marginals $m_1$ and $m_2$. The space $(\pr_2(\br),W_2)$ is a complete separable metric space.  Let $\mathcal{M}_2(C[0,T])$ denote the space of all probability measures $m$ on $C[0,T]$ such that
	\begin{equation*}
		\m_2(m):=\int\sup_{0\le t\le T}|x(t)|^2dm(x)<\infty.
	\end{equation*}
	The measure on it is defined by
	\begin{equation*}
		\begin{split}
			D_2(m_1,m_2)=\big(\inf_{\gamma\in\Gamma(m_1,m_2)}\int_{\br^2} \sup_{0\le t\le T} |x(t,\omega_1)-x(t,\omega_2)|^2 d\gamma(\omega_1,\omega_2)\big)^{\frac{1}{2}},\quad m_1,m_2\in\mathcal{M}_2(C[0,T]).
		\end{split}
	\end{equation*}
	The space $(\mathcal{M}_2(C[0,T]),D_2)$ is a complete separable metric space.
	
	\subsection{$N$-player Stochastic Differential Games}
	Let $T>0$ be a fixed terminal time, $\tilde{W}=\{\tilde{W}_t,0\le t\le T\}$ and $W^i=\{W^i_t,0\le t\le T\}$, $i=1,2,\dots,N$ are one-dimensional independent Brownian motions defined on a complete probability space $(\Omega,\mathbb{P})$ satisfying the usual conditions. Consider a stochastic dynamic game of $N$ players. The $i$-th player regulates his/her own state process $X_t^i$ in $\br$ governed by
	\begin{equation*}
		\left\{
		\begin{aligned}
			&dX_t^i=b(t,X_t^i,u_t^i,m_t^N)dt+\sigma(t,X_t^i,u_t^i,m_t^N)dW_t^i+\tilde{\sigma}(t,X_t^i,u_t^i,m_t^N)d\tilde{W}_t, \quad t\in(0,T],\quad 1\le i\le N;\\
			&X_0^i=\xi_0^i,
		\end{aligned}
		\right.
	\end{equation*}
	via the control process $u^i=\{u_t^i,0\le t\le T\}\in\lr_{\f^i}^2(0,T)$, where
	\begin{equation*}
		b,\sigma,\tilde{\sigma}:[0,T]\times\br\times\br\times\pr_2(\br)\to\br,
	\end{equation*}
	$\mathscr{F}^i$ is the natural filtration of $(\xi_0^i,W^i,\tilde{W})$, and $m_t^N$ is the empirical distribution of $\{X_t^i,1\le i\le N\}$, i.e.
	\begin{equation*}
		m_t^N=\frac{1}{N}\sum_{i=1}^N\delta_{X_t^i}(dx),\quad t\in[0,T].
	\end{equation*}
	We assume that $\{\xi^i_0,1\le i\le N\}$ are independent and identically distributed, independent of all Brownian motions and satisfy $\e[|\xi_0^i|^2]<\infty$ for all $1\le i\le N$. We call $\tilde{W}$ a common noise and $W^i$ an individual noise of the $i$-th player.
	
	Given the other players' strategies, the $i$-th player selects a control $u^i\in\lr_{\f^i}^2(0,T)$ to minimize his/her expected cost 
	\begin{equation*}
		J^i(u^i|(u^j)_{j\neq i}):=\mathbb{E}\big[\int_0^T f(t,X_t^i,u_t^i,m_t^N)dt+g(X^i_T,m_T^N)\big]
	\end{equation*}
	where $(u^j)_{j\neq i}$ denotes a strategy profile of other players excluding the $i$-th player, and 
	\begin{equation*}
		\begin{split}
			&f:[0,T]\times\br\times\br\times\pr_2(\br)\to\br,\quad g:\br\times\pr_2(\br)\to\br,
		\end{split}
	\end{equation*}
	are assumed to be identical for all players.
	
	Note that the strategies of other players have an effect on the $i$-th player through $m_t^N$, which is the main feature that makes this set up a game. We are seeking a type of equilibrium solution widely used in game theory setting called Nash equilibrium.
	\begin{definition}
		A set of strategies $(u^i)_{1\le i\le N}$ is a Nash Equilibrium if $u^i$ is optimal for the $i$-th player given the other players' strategies $(u^j)_{j\neq i}$. In other words,
		\begin{equation*}
			J^i(u^i|(u^j)_{j\neq i})=\min_{u\in\lr_{\f^i}^2(0,T)}J^i(u|(u^j)_{j\neq i}), \quad 1\le i\le N.
		\end{equation*}
	\end{definition}
	Solving for a Nash equilibrium of an $N$-player game is impractical when $N$ is large due to the curse of dimensionality. So we formally take the limit as $N\to\infty$ and consider the limit problem instead.
	
	\subsection{Formulation of the Problem}
	We now formulate the MFG with a common noise by taking the limit of $N$-player stochastic differential games as $N\to\infty$. When considering the limiting problem, we assume that each player adopts the same strategy. Therefore, the players' distribution can be represented by a conditional law of a single representative player given a common noise. In other words, we formulate the MFG with a common noise as a stochastic control problem for a single player with an equilibrium condition involving a conditional law of the state process given a common noise.
	
	Let $T>0$ be a fixed terminal time, $W=\{W_t,0\le t\le T\}$ and $\tilde{W}=\{\tilde{W}_t,0\le t\le T\}$ be one-dimensional independent Brownian motions defined on a complete probability space $(\Omega,\mathbb{P})$ satisfying the usual conditions. Let $\xi_0$ be a square-integrable random variable. We assume that $\f=\{\f_t,0\le t\le T\}$ is the natural filtration of $(\xi_0,W,\tilde{W})$ and $\tilde{\f}=\{\tilde{\f}_t,0\le t\le T\}$ is the natural filtration of $\tilde{W}$. Both of them are augmented by all the $\mathbb{P}$-null sets. The problem of MFG with a common noise is defined as follows:
	
	\begin{problem}\label{def1}
		For given measurable functions $b,\sigma,\tilde{\sigma},f:[0,T]\times\br\times\br\times\pr_2(\br)\to\br$ and $g:\br\times\pr_2(\br)\to\br$, find an optimal control $\hat{u}\in \mathcal{L}_{\mathscr{F}}^2(0,T)$ for the stochastic control problem
		\begin{equation*}
			\left\{
			\begin{aligned}
				&\hat{u}\in\argmin_{u\in\mathcal{L}_{\mathscr{F}}^2(0,T)}J(u|m):=\e\big[\int_0^T f(t,X_t^u,u_t,m_t)dt+g(X^u_T,m_T)\big];\\
				&X_t^u=\xi_0+\int_0^tb(s,X_s^u,u_s,m_s)ds+\int_0^t\sigma(s,X_s^u,u_s,m_s)dW_s\\
				&\qquad\quad+\int_0^t\tilde{\sigma}(s,X_s^u,u_s,m_s)d\tilde{W}_s,\quad t\in[0,T];\\
				&m_t=\mathcal{L}(X_t^{\hat{u}}|\tilde{\mathscr{F}}_t),\quad \xi_0\in\lr_{\f_0}^2.
			\end{aligned}
			\right.
		\end{equation*}
	\end{problem}
	
	\section{Stochastic Maximum Principle}\label{SMP}
	In this section, we discuss the stochastic maximum principle for MFG with a common noise. The stochastic maximum principle gives optimality conditions satisfied by an optimal control. It gives sufficient and necessary conditions for the existence of an optimal control in terms of solvability of the adjoint process as a backward stochastic differential equation (BSDE). For more details about stochastic maximum principle, we refer to \cite{HP} or \cite{JYXY}. In our case, Problem~\ref{def1} is associated to a conditional distribution dependent FBSDE with the help of the sufficient Pontryagin principle for optimality.
	
	We begin with discussing the stochastic maximum principle given an $\tilde{\f}_t$-progressively-measurable stochastic flow of probability measures $m=\{m_t,0\le t\le T\}\in\mathcal{M}_2(C[0,T])$. We define the generalized Hamiltonian
	\begin{equation*}
		\begin{split}
			H(t,x,p,q,\tilde{q},u,m):=b(t,x,u,m)p+\sigma(t,x,u,m)q+\tilde{\sigma}(t,x,u,m)\tilde{q}+f(t,x,u,m),\\
			%t\in[0,T],\quad x,p,q,\tilde{q},u\in\br,\quad m\in\pr_2(\br).\\
			(t,x,p,q,\tilde{q},u,m)\in [0,T]\times\mathbb{R}^5\times\pr_2(\br).
		\end{split}
	\end{equation*}
	Now we state the first set of assumptions to ensure that the stochastic control problem is uniquely solvable given $m$. For notational convenience, we use the same constant $L$ for all the conditions below.
	
	\textbf{(H1)} The drift $b$ and the volatility $\sigma,\tilde{\sigma}$ are linear in $x$ and $u$. They read
	\begin{equation*}
		\phi(t,x,u,m)=\phi_0(t,m)+\phi_1(t)x+\phi_2(t)u,\quad \phi=b,\sigma,\tilde{\sigma},\quad \phi_i=b_i,\sigma_i,\tilde{\sigma}_i,
	\end{equation*}
	for some measurable deterministic functions $\phi_0:[0,T]\times\pr_2(\br)\to\br$ satisfing the following linear growth:
	\begin{equation*}
		|\phi_0(t,m)|\le L\big(1+\big(\int_{\br}|x|^2dm(x)\big)^{\frac{1}{2}}\big),
	\end{equation*} 
	and $\phi_1,\phi_2:[0,T]\to\br$ being bounded by a positive constant $L$. Further, $(\sigma_2, \tilde{\sigma}_2)$ is bounded by a positive constant $B_u$. For notational convenience, we can assume that $B_u\le L$ by setting $L=\max\{L, B_u\}$.
	
	\textbf{(H2)} The function $f(t,0,0,m)$ satisfies a quadratic growth condition in $m$. The function $f(t,\cdot,\cdot,m):\br\times\br\to\br$ is differentiable for all $(t,m)\in[0,T]\times\pr_2(\br)$, with the derivatives $(f_x,f_u)(t,x,u,m)$ satisfying a linear growth in $(x,u,m)$. Similarly, the function $g(0,m)$ satisfies a quadratic growth condition in $m$. The function $g(\cdot,m):\br\to\br$ is differentiable for all $m\in\pr_2(\br)$, with the derivative $g_x(x,m)$ satisfying a linear growth in $(x,m)$. That is, 
	\begin{align*}
		&\max\{|f(t,0,0,m)|,|g(0,m)|\}\le L\big(1+\int_{\br}|x|^2dm(x)\big),\quad m\in\pr_2(\br);\\
		&\max\{|f_x(t,x,u,m)|,|f_u(t,x,u,m)|,|g_x(x,m)|\}\le L\big(1+|x|+|u|+\big(\int_{\br}|x|^2dm(x)\big)^{\frac{1}{2}}\big),\\
		&\qquad\qquad\qquad\qquad\qquad\qquad\qquad\qquad\qquad\qquad\quad(t,x,u,m)\in[0,T]\times\br\times\br\times\pr_2(\br).
	\end{align*}
	
	\textbf{(H3)} The function $f$ is of the form 
	\begin{equation*}
		f(t,x,u,m)=f_0(t,x,u)+f_1(t,x,m),\quad (t,x,u,m)\in[0,T]\times\br\times\br\times\pr_2(\br).
	\end{equation*}
	The function $f_0$ is differentiable with respect to $(x,u)$ and the function $f_1$ is differentiable with respect to $x$. The derivatives $(f_{0x},f_{0u})(t,\cdot,\cdot):\br\times\br\to\br\times\br$ are $L$-Lipschitz continuous uniformly in $t\in[0,T]$. The derivative $f_{1x}(t,\cdot,m):\br\to\br$ is $L$-Lipschitz continuous uniformly in $(t,m)\in[0,T]\times\pr_2(\br)$. The derivative $g_{x}(\cdot,m):\br\to\br$ is $L$-Lipschitz continuous uniformly in $m\in\pr_2(\br)$. 
	
	\textbf{(H4)} The functions $f_1(t,\cdot,m)$ and $g(\cdot,m)$ are convex for all $(t,m)\in[0,T]\times\pr_2(\br)$, in such a way that
	\begin{equation*}
		\begin{split}
			&(f_{1x}(t,x',m)-f_{1x}(t,x,m))(x'-x)\geq0,\quad t\in[0,T],\quad x,x'\in\br,\quad m\in\pr_2(\br);\\
			&(g_x(x',m)-g_x(x,m))(x'-x)\geq0,\quad x,x'\in\br,\quad m\in\pr_2(\br).
		\end{split}
	\end{equation*}
	The function $f_0(t,x,u)$ is jointly convex in $(x,u)$ with a strict convexity in $u$ for all $t\in[0,T]$, in such a way that, for some $C_f>0$,
	\begin{equation*}
		\begin{split}
			f_0(t,x',u')-f_0(t,x,u)-(f_{0x},f_{0u})(t,x,u)\cdot(x'-x,u'-u)\geq C_f|u'-u|^2, \\
			t\in[0,T], \quad x,x',u,u'\in\br.
		\end{split}
	\end{equation*}
	
	The linear growth condition (H2) and Lipschitz condition (H3) are standard assumptions to ensure the existence of a strong solution. The linear-convex conditions (H1) and (H4) ensure that the Hamiltonian is strictly convex, so that there is a unique minimizer in the feedback form. The separability condition in (H3) ensures that the feedback control is independent of $m$. The following result is borrowed from \cite[Lemma 1]{PA}.
	
	\begin{lemma}\label{lemma:u}
		Under assumptions (H1)-(H4), given $m\in\pr_2(\br)$, for all $(t,x,p,q,\tilde{q})\in [0,T]\times \mathbb{R}\times\mathbb{R}\times\mathbb{R}\times\mathbb{R}$, there exists a unique minimizer $\hat{u}(t,x,p,q,\tilde{q})$ of the Hamiltonian. Moreover, the map $(t,x,p,q,\tilde{q})\mapsto \hat{u}(t,x,p,q,\tilde{q})$ is measurable, locally bounded and $L(2C_f)^{-1}$-Lipschitz continuous in $(x,p)$ and $B_u(2C_f)^{-1}$-Lipschitz continuous in $(q,\tilde{q})$, uniformly in $t\in [0,T]$. 
	\end{lemma}
	
	In fact, under assumptions (H1)-(H4), if we take derivative of $H$ with respect to $u$, we know that $\hat{u}(t,x,p,q,\tilde{q})$ satisfies 
	\begin{equation}\label{remark:1}
		b_2(t)p+\sigma_2(t)q+\tilde{\sigma}_2(t)\tilde{q}+f_{0u}(t,x,\hat{u}(t,x,p,q,\tilde{q}))=0,\quad t\in[0,T],\quad x,p,q,\tilde{q}\in\br.
	\end{equation}
	We know from Lemma~\ref{lemma:u} that $\hat{u}$ is Lipschitz continuous with respect to $(x,p,q,\tilde{q})$ uniformly in $t\in [0,T]$. We define
	\begin{equation*}
		\hat{u}^0(t):=\hat{u}(t,0,0,0,0),\quad t\in[0,T].
	\end{equation*}
	Now we give a bound of $\hat{u}^0(t)$. From \eqref{remark:1} we know that $\hat{u}^0(t)$ satisfies
	\begin{equation*}
		f_{0u}(t,0,\hat{u}^0(t))=0.
	\end{equation*}
	Using the concex assumption (H4), we have that
	\begin{align*}
		&f_0(t,0,\hat{u}^0(t))-f_0(t,0,0)-f_{0u}(t,0,0)\hat{u}^0(t)\geq C_f|\hat{u}^0(t)|^2,\\
		&f_0(t,0,0)-f_0(t,0,\hat{u}^0(t))+f_{0u}(t,0,\hat{u}^0(t))\hat{u}^0(t)\geq C_f|\hat{u}^0(t)|^2,
	\end{align*}
	which imply
	\begin{equation*}
		-f_{0u}(t,0,0)\hat{u}^0(t)\geq 2C_f|\hat{u}^0(t)|^2.
	\end{equation*}
	The above estimate and assumption (H2) show that 
	\begin{equation}\label{bound}
		|\hat{u}^0(t)|\le L(2C_f)^{-1}.
	\end{equation}
	
	We are ready to state the stochastic maximum principle for a given stochastic flow of probability measures $m=\{m_t,0\le t\le T\}\in\mathcal{M}_2(C[0,T])$. We define the control problem $\mathscr{P}^m$:
	\begin{equation*}
		\left\{
		\begin{aligned}
			&\hat{u}\in\argmin_{u\in\mathcal{L}_{\mathscr{F}}^2(0,T)}\e\big[\int_0^T f(t,X_t^u,u_t,m_t)dt+g(X_T^u,m_T)\big],\\
			&X_t^u=\xi_0+\int_0^tb(s,X_s^u,u_s,m_s)ds+\int_0^t\sigma(s,X_s^u,u_s,m_s)dW_s+\int_0^t\tilde{\sigma}(s,X_s^u,u_s,m_s)d\tilde{W}_s,\\
			&\qquad\qquad\qquad\qquad\qquad\qquad\qquad\qquad\qquad\qquad\qquad\qquad\qquad\quad t\in(0,T],\quad \xi_0\in\lr_{\f_0}^2.
		\end{aligned}
		\right.
	\end{equation*}
	Note that $\mathscr{P}^m$ is a classical control problem with the random coefficients $(b,\sigma,\tilde{\sigma},f)(t,\cdot,\cdot,m_t)$ for $t\in[0,T]$ and $g(\cdot,m_T)$. We have the following stochastic maximum principle. 
	
	\begin{theorem}\label{thm:mp}
		Suppose that assumptions (H1)-(H4) hold. For a given flow of probability measures $m=\{m_t,0\le t\le T\}\in\mathcal{M}_2(C[0,T])$ and $\xi_0\in\lr^2_{\f_0}$, if the following FBSDE
		\begin{equation}\label{FBSDE^m}
			\left\{	
			\begin{aligned}
				&dX_t=b(t,X_t,\hat{u}(t,X_t,p_t,q_t,\tilde{q}_t),m_t)dt+\sigma(t,X_t,\hat{u}(t,X_t,p_t,q_t,\tilde{q}_t),m_t)dW_t\\
				&\qquad\quad+\tilde{\sigma}(t,X_t,\hat{u}(t,X_t,p_t,q_t,\tilde{q}_t),m_t)d\tilde{W}_t,\quad t\in(0,T];\\
				&dp_t=-\partial_x H(t,X_t,p_t,q_t,\tilde{q}_t,\hat{u}(t,X_t,p_t,q_t,\tilde{q}_t),m_t)dt+q_tdW_t+\tilde{q}_td\tilde{W}_t,\quad t\in[0,T);\\
				&X_0=\xi_0,\quad p_T=g_x(X_T,m_T)
			\end{aligned}
			\right.
		\end{equation}
		has an adapted solution $\{(\hat{X}_t,\hat{p}_t,\hat{q}_t,\hat{\tilde{q}}_t),0\le t\le T\}$ such that 
		\begin{equation}\label{estimate1}
			\mathbb{E}\big[\sup_{0\le t\le T}|(\hat{X}_t,\hat{p}_t)|^2+\int_0^T|(\hat{q}_t,\hat{\tilde{q}}_t)|^2dt \big]<\infty,
		\end{equation}
		then, $\hat{u}=\{\hat{u}(t,\hat{X}_t,\hat{p}_t,\hat{q}_t,\hat{\tilde{q}}_t),0\le t\le T\}$ is an optimal control of the control problem $\mathscr{P}^m$. Furthermore, for any $u\in \lr_{\f}^2(0,T)$, we have the following estimate 
		\begin{equation}\label{estimate}
			J(\hat{u}|m)+C_f \mathbb{E}\big[\int_0^T|u_t-\hat{u}_t|^2dt\big]\le  J(u|m).
		\end{equation}
		In particular, $\hat{u}$ is the unique optimal control.
	\end{theorem}
	
	\begin{proof}
		The proof is standard and we refer to \cite[Theorem 6.4.6]{HP}. The estimate \eqref{estimate} requires strict convexity in $u$ of $f_0$. The proof can be found in \cite[Theorem 2.2]{PA}.
	\end{proof}
	
	We now show FBSDE \eqref{FBSDE^m} is uniquely solvable, which implies that problem $\mathscr{P}^m$ is uniquely solvable. We state the slightly more general result for a random terminal function and an arbitrary initial and terminal time, which will arise in a subsequent section. It is an immediate consequence of \cite[Theorem 2.3]{SP}, concerning the existence and uniqueness of a solution to a monotone FBSDE.
	
	\begin{theorem}\label{thm:mp'}
		Let $0\le s< \tau\le T$ and $\xi\in\lr_{\f_s}^2$. Suppose that (H1)-(H4) hold. Suppose that $v:\br\times\Omega\to\br$ is an $\f_\tau$-measurable $C_v$-Lipschitz continuous function satisfying the following  monotonicity condition
		\begin{equation}\label{fbsde_v}
			(v(x',\omega)-v(x,\omega))(x'-x)\geq 0,\quad x,x'\in\br,\quad \omega\in\Omega.
		\end{equation}
		Then, for a given flow of probability measures $m=\{m_t,s\le t\le \tau\}\in\mathcal{M}_2(C[s,\tau])$, there exists a unique adapted solution $\{(X_t,p_t,q_t,\tilde{q}_t),s\le t\le \tau\}$ to FBSDE
		\begin{equation*}\label{fbsde_equation}
			\left\{
			\begin{aligned}
				&dX_t=b(t,X_t,\hat{u}(t,X_t,p_t,q_t,\tilde{q}_t),m_t)dt+\sigma(t,X_t,\hat{u}(t,X_t,p_t,q_t,\tilde{q}_t),m_t)dW_t\\
				&\qquad\quad+\tilde{\sigma}(t,X_t,\hat{u}(t,X_t,p_t,q_t,\tilde{q}_t),m_t)d\tilde{W}_t,\quad t\in(s,\tau];\\
				&dp_t=-\partial_x H(t,X_t,p_t,q_t,\tilde{q}_t,\hat{u}(t,X_t,p_t,q_t,\tilde{q}_t),m_t)dt+q_tdW_t+\tilde{q}_td\tilde{W}_t,\quad t\in[s,\tau);\\
				&X_s=\xi,\quad p_{\tau}=v(X_{\tau})
			\end{aligned}
			\right.
		\end{equation*}
		such that
		\begin{equation}\label{fbsde_estimate}
			\begin{split}
				&\e\big[\sup_{s\le t\le \tau}|(\hat{X}_t,\hat{p}_t)|^2+\int_s^{\tau}|(\hat{q}_t,\hat{\tilde{q}}_t)|^2dt\big]\le C\big(\e[|\xi|^2+|v(0)|^2+\m_2(m)]+1\big) 
			\end{split}
		\end{equation}
		for some constant $C$ depending on $(L,T,C_f,C_v)$.
	\end{theorem}
	
	Note that this theorem implies that there exist a unique adapted solution $\{(X_t,p_t,q_t,\tilde{q}_t),0\le t\le T\}$ to FBSDE \eqref{FBSDE^m} by setting $s=0$, $\tau=T$  and $v(x)=g_x(x,m_T)$. Assumptions (H3) and (H4) ensure that $g_x(x,m_T)$ is $L$-Lipschitz continuous with respect to $x$ and satisfies the monotonicity condition \eqref{fbsde_v}. If we further assume that $\e[\m_2(m)]<\infty$, then from \eqref{fbsde_estimate} and the linear growth assumption (H2), we have the estimate \eqref{estimate1}. Then, as a consequence of Theorems~\ref{thm:mp}, problem $\mathscr{P}^m$ has a unique optimal control given by $\hat{u}_t=\hat{u}(t,X_t,p_t,q_t,\tilde{q}_t)$.
	
	Now we turn to Problem~\ref{def1}. It states that given the stochastic flow of probability measures $m\in\m_2(C[0,T])$, the state process $X^{u^m}$ corresponding to the optimal control $u^m$ of the problem $\mathscr{P}^m$ satisfies the following consistency
	\begin{equation*}
		m_t=\lr(X_t^{u^m}|\tilde{\f}_t).
	\end{equation*}
	Plugging this into Theorem~\ref{thm:mp}, we have the stochastic maximum principle for Problem~\ref{def1}.
	\begin{theorem}
		Suppose that assumptions (H1)-(H4) hold. For $\xi_0\in\lr_{\f_0}^2$, if the following FBSDE	
		\begin{equation}\label{fbsde1}
			\left\{
			\begin{aligned}
				&dX_t=b(t,X_t,\hat{u}(t,X_t,p_t,q_t,\tilde{q}_t),\lr({X_t|\tilde{\f}_t}))dt+\sigma(t,X_t,\hat{u}(t,X_t,p_t,q_t,\tilde{q}_t),\lr({X_t|\tilde{\f}_t}))dW_t\\
				&\quad\quad\quad+\tilde{\sigma}(t,X_t,\hat{u}(t,X_t,p_t,q_t,\tilde{q}_t),\lr({X_t|\tilde{\f}_t}))d\tilde{W}_t,\quad t\in(0,T];\\
				&dp_t=-\partial_x H(t,X_t,p_t,q_t,\tilde{q}_t,\hat{u}(t,X_t,p_t,q_t,\tilde{q}_t),\lr({X_t|\tilde{\f}_t}))dt+q_tdW_t+\tilde{q}_td\tilde{W}_t,\quad t\in[0,T);\\
				&X_0=\xi_0,\quad p_T=g_x(X_T,\lr({X_T|\tilde{\f}_T}))
			\end{aligned}
			\right.
		\end{equation}
		has an adapted solution $\{(\hat{X}_t,\hat{p}_t,\hat{q}_t,\hat{\tilde{q}}_t),0\le t\le T\}$
		such that 
		\begin{equation}\label{space}
			\mathbb{E}\big[\sup_{0\le t\le T}|(\hat{X}_t,\hat{p}_t)|^2+\int_0^T|(\hat{q}_t,\hat{\tilde{q}}_t)|^2dt \big]<\infty,
		\end{equation}
		then, $\hat{u}_t=\hat{u}(t,\hat{X}_t,\hat{p}_t,\hat{q}_t,\hat{\tilde{q}}_t)$ is an optimal control of Problem~\ref{def1}.
	\end{theorem}
	
	In the rest of this paper, we discuss the existence and uniqueness of the solution of FBSDE \eqref{fbsde1}. We will give two different methods under the following additional conditions on the dependence of $(b_0,\sigma_0,\tilde{\sigma}_0,f_{1x},g_x)$ on the measure variable $m$. 
	
	\textbf{(H5)} (Lipschitz continuity in $m$) The functions $(b_0,\sigma_0,\tilde{\sigma}_0)(t,\cdot):\pr_2(\br)\to\br$ are $L_m$-Lipschitz continuous uniformly in $t\in[0,T]$. The function $f_{1x}(t,x,\cdot):\pr_2(\br)\to\br$ is $L$-Lipschitz continuous uniformly in $(t,x)\in[0,T]\times\br$. The function $g_x(x,\cdot):\pr_2(\br)\to\br$ is $L$-Lipschitz continuous uniformly in $x\in\br$. That is,
	\begin{equation*}
		\begin{split}
			&|b_0(t,m')-b_0(t,m)|+|\sigma_0(t,m')-\sigma_0(t,m)|+|\tilde{\sigma}_0(t,m')-\tilde{\sigma}_0(t,m)|\\
			&\qquad\le L_m W_2(m,m'),\quad t\in[0,T],\quad m,m'\in \pr_2(\br);\\
			&|f_{1x}(t,x,m')-f_{1x}(t,x,m)|+|g_x(x,m')-g_x(x,m)|\\
			&\qquad\le L W_2(m,m'),\quad t\in[0,T], \quad x\in \br,\quad m,m'\in \pr_2(\br).
		\end{split}
	\end{equation*}
	For notational convenience, we assume that $L_m\le L$ by setting $L=\max\{L, L_m\}$.
	
	\textbf{(H6)} (Weak monotonicity in $m$) For any $\gamma\in \pr_2(\br^2)$ with marginals $m$ and $m'$, 
	\begin{equation*}
		\begin{split}
			&\int_{\br^2}[(f_{1x}(t,x,m)-f_{1x}(t,y,m'))(x-y)]\gamma(dx,dy)\geq 0, \quad t\in[0,T],\quad m,m'\in\pr_2(\br);\\
			&\int_{\br^2}[(g_x(x,m)-g_x(y,m'))(x-y)]\gamma(dx,dy)\geq 0,\quad m,m'\in\pr_2(\br).    
		\end{split}
	\end{equation*}
	Equivalently, for any square-integrable random variables $\xi$ and $\xi'$ on the same probability space,
	\begin{equation*}
		\begin{split}
			&\e [(f_{1x}(t,\xi',\lr(\xi'))-f_{1x}(t,\xi,\lr(\xi)))(\xi'-\xi)]\geq 0,\quad t\in[0,T];\\
			&\e [(g_x(\xi',\lr(\xi'))-g_x(\xi,\lr(\xi)))(\xi'-\xi)]\geq 0.    
		\end{split}
	\end{equation*}
	
	\section{Solvability of FBSDE \eqref{fbsde1}: Method One}\label{M1}
	In this section, we give the existence and uniqueness result of the solution to FBSDE \eqref{fbsde1} by the method of continuation in coefficients. We have the following main result.
	\begin{theorem}\label{main1_thm}
		Suppose that assumptions (H1)-(H6) hold and $\xi_0\in\mathcal{L}_{{\mathscr{F}}_0}^2$. There exists $\delta>0$ depending only on $(L,T)$ such that FBSDE \eqref{fbsde1} is uniquely solvable when $L_mC_f^{-1}\le\delta$.
	\end{theorem}
	
	In this section, we always suppose that assumptions (H1)-(H6) hold. To prove the above theorem, a natural and simple strategy consists in modifying the coefficients in a linear way and proving that there is the existence and uniqueness when coefficients in the FBSDE are slightly perturbed. To avoide heavy notations, we use the following conventions. The notation $\{\theta_t,0\le t\le T\}$ denotes a process $\{(X_t,u_t,m_t),0\le t\le T\}$ with $m_t=\lr(X_t|\tilde{\f}_t)$. The notation $\{\Theta_t,0\le t\le T\}$ stands for a process of the form $\{(\theta_t,p_t,q_t,\tilde{q}_t),0\le t\le T\}$. We denote by $\mathbb{S}$ the space of processes $\{\Theta_t,0\le t\le T\}$ such that $\{(X_t,u_t,p_t,q_t,\tilde{q}_t),0\le t\le T\}$ is $\f_t$-progressively-measurable, and
	\begin{equation*}
		\|\Theta\|_{\mathbb{S}}:=\big(\e\big[\sup_{0\le t\le T}|( X_t,p_t)|^2+\int_0^T|(u_t,q_t,\tilde{q}_t)|^2dt\big]\big)^{\frac{1}{2}}<+\infty.
	\end{equation*}
	
	We call an input for FBSDE \eqref{fbsde1} a five-tuple
	\begin{equation*}
		\mathcal{I}=((\ii_t^b,\ii_t^{\sigma},\ii_t^{\tilde{\sigma}},\ii_t^f)_{0\le t\le T},\ii_T^g)
	\end{equation*}
	with  $(\ii_t^b,\ii_t^{\sigma},\ii_t^{\tilde{\sigma}},\ii_t^f)_{0\le t\le T}$ being four square-integrable progressively-measurable processes and $\ii_T^g$ being a square-integrable $\f_T$-measurable random variable. Such an input is specifically designed to be injected into the dynamics of FBSDE \eqref{fbsde1}. Denote by $\mathbb{I}$ the space of all inputs, endowed with the norm
	\begin{equation*}
		\|\ii\|_{\iii}:=\e\bigg[|\ii_T^g|^2+\int_0^T |(\ii_t^b,\ii_t^{\sigma},\ii_t^{\tilde{\sigma}},\ii_t^f)|^2 dt\bigg]^{\frac{1}{2}}.
	\end{equation*}
	
	\begin{definition}
		For any $(\gamma,\xi)\in[0,1]\times \lr_{\f_0}^2$ and any input $\ii\in\iii$, denote by $\mathcal{E}(\gamma,\xi,\ii)$ the FBSDE
		\begin{equation}\label{main2}
			\left\{
			\begin{aligned}
				&dX_t=(\gamma b(t,\theta_t)+\ii_t^b)dt+(\gamma\sigma(t,\theta_t)+\ii_t^{\sigma})dW_t+(\gamma\tilde{\sigma}(t,\theta_t)+\ii_t^{\tilde{\sigma}})d\tilde{W}_t,\quad t\in(0,T];\\
				&dp_t=-(\gamma H_x(t,\Theta_t)+\ii_t^f)dt+q_tdW_t+\tilde{q}_td\tilde{W}_t,\quad t\in[0,T),
			\end{aligned}
			\right.
		\end{equation}
		with $m_t=\lr(X_t|\tilde{\f}_t)$, 
		\begin{equation}\label{main3}
			u_t=\hat{u}(t,X_t,p_t,q_t,\tilde{q}_t),\quad t\in[0,T],
		\end{equation}
		$X_0=\xi$ (initial condition) and $p_T=\gamma g_x(X_T,m_T)+\ii_T^g$ (terminal condition).
		%\begin{equation}\label{main4}
		%p_T=\gamma g_x(X_T,m_T)+\ii_T^g
		%\end{equation}
	\end{definition}
	
	Whenever $\{(X_t,p_t,q_t,\tilde{q}_t),0\le t\le T\}$ is a solution, $\{(X_t,u_t,m_t,p_t,q_t,\tilde{q}_t),0\le t\le T\}$ is referred to as the associated extended solution. Note that the forward and the backward equations in \eqref{main2} are coupled via the optimality condition \eqref{main3}. When $\gamma=1$, $\ii=0$ and $\xi=\xi_0$, the pair \eqref{main2}-\eqref{main3} coincides with FBSDE \eqref{fbsde1}. For notational convenience, we set the following deninition.
	
	\begin{definition}
		Given $\gamma\in [0,1]$, we say that property $(\mathcal{S}_\gamma)$ holds true if, for any $\xi\in \lr_{\f_0}^2$ and any $\ii\in\iii$, FBSDE $\mathcal{E}(\gamma,\xi,\ii)$ has a unique extended solution in $\s$.
	\end{definition}
	
	Our aim is to show $(\mathcal{S}_1)$ holds true. The following lemma is proved in Appendix~\ref{proof1}.
	
	\begin{lemma}\label{main1_lem1}
		Suppose that assumptions (H1)-(H6) hold. Let $\gamma\in [0,1]$ such that $(\mathcal{S}_\gamma)$ holds true. Then, there exists $\delta>0$ depending only on $(L,T)$, and a constant $C$ independent of $\gamma$, such that for any $\xi^1,\xi^2\in \lr_{\f_0}^2$ and $\ii^1,\ii^2\in\iii$, the respective extended solutions $\Theta^1$ and $\Theta^2$ of $\mathcal{E}(\gamma,\xi^1,\ii^1)$ and $\mathcal{E}(\gamma,\xi^2,\ii^2)$ satisfy
		\begin{equation*}
			\|\Theta^1-\Theta^2\|_{\mathbb{S}}\le C([\e|\xi^1-\xi^2|^2]^{\frac{1}{2}}+\|\ii^1-\ii^2\|_{\iii}),
		\end{equation*}
		when $L_mC_f^{-1}\le\delta$.
	\end{lemma}
	
	We now give the following lemma, which plays a crucial role in the proof of Theorem~\ref{main1_thm}.
	
	\begin{lemma}\label{main1_lem2}
		Suppose that assumptions (H1)-(H6) hold. There exist $\delta>0$ depending only on $(L,T)$ and $\eta_0>0$ such that,  if $L_mC_f^{-1}\le\delta$ and $(\mathcal{S}_{\gamma})$ holds true for some $\gamma\in[0,1)$, then $(\mathcal{S}_{\gamma+\eta})$ holds true for any $\eta\in(0,\eta_0]$ satisfying $\gamma+\eta\le 1$.
	\end{lemma}
	
	\begin{proof}
		The proof follows from the contraction of Picard's mapping. Consider $\gamma$ such that $(\mathcal{S}_{\gamma})$ holds true. For $\eta>0$, any $\xi\in \lr_{\f_0}^2$ and any $\ii\in\iii$, we aim to show that the FBSDE $\mathcal{E}(\gamma+\eta,\xi,\ii)$ has a unique extended solution in $\mathbb{S}$. To do so, we define a map $\Phi:\mathbb{S}\to\mathbb{S}$, whose fixed points are solutions of $\mathcal{E}(\gamma+\eta,\xi,\ii)$. 
		
		The definition $\Phi$ is as follows. Given a process $\Theta\in\mathbb{S}$, we denote by $\Theta'$ the extended solition of the FBSDE $\mathcal{E}(\gamma,\xi,\ii')$ with
		\begin{equation*}
			\begin{split}
				&\ii_t^{b,'}=\eta b(t,\theta_t)+\ii_t^b,\quad t\in[0,T];\\
				&\ii_t^{\sigma,'}=\eta\sigma(t,\theta_t)+\ii_t^{\sigma},\quad t\in[0,T];\\
				&\ii_t^{\tilde{\sigma},'}=\eta\sigma(t,\theta_t)+\ii_t^{\tilde{\sigma}},\quad t\in[0,T];\\
				&\ii_t^{f,'}=\eta H_x(t,\Theta_t)+\ii_t^f,\quad t\in[0,T];\\
				&\ii_T^{g,'}=\eta g_x(X_T,m_T)+\ii_T^g.
			\end{split}
		\end{equation*}
		From the assumption that $(\mathcal{S}_{\gamma})$ holds for true, $\Theta'$ is uniquely defined, and it belongs to $\mathbb{S}$, so that $\Phi:\Theta\mapsto\Theta'$ maps $\mathbb{S}$ into itself. It is then clear that a process $\Theta\in\mathbb{S}$ is a fixed point of $\Phi$ if and only if $\Theta$ is an extended solution of $\mathcal{E}(\gamma+\eta,\xi,\ii)$. So we only need to illustrate that $\Phi$ is a contraction when $\eta$ is small enough. 
		
		In fact, for any $\Theta^1,\Theta^2\in\mathbb{S}$, we know from Lemma~\ref{main1_lem1} that
		\begin{equation*}
			\|\Phi(\Theta^{2})-\Phi(\Theta^1)\|_{\mathbb{S}}\le C\|\ii'^{2}-\ii'^{1}\|_{\iii}\le C\eta\|\Theta^2-\Theta^2\|_{\mathbb{S}},
		\end{equation*}
		where $C$ is independent of $\gamma$ and $\eta$. So when $\eta$ is small enough, $\Phi$ is indeed a contraction.
	\end{proof}
	
	\emph{Proof of Theorem~\ref{main1_thm}.}
	In view of Lemma~\ref{main1_lem2}, we only need to prove that $(\mathcal{S}_0)$ holds true, which is obviously true since there is no coupling between the forward and the backward equations when $\gamma=0$.
	\endproof

	\section{Solvability of FBSDE \eqref{fbsde1}: Method Two}\label{M2}
	
	In this section, we prove the existence and uniqueness of the solution of FBSDE \eqref{fbsde1} with an alternative method. In the first subsection, we use the weak monotonicity assumption to deduce the uniqueness result. And in the second subsection, we first show the existence result on a small time interval $[\tau,T]$ and then extend the local solution to the whole time interval $[0,T]$. More assumptions are required than the first method. However, the intermediate result can better demonstrate the probabilistic properties as well as the sensitivity, which are worthy of study.
	
	\subsection{Uniqueness}
	We have the following uniqueness of the solution of FBSDE \eqref{fbsde1}.
	\begin{theorem}[Uniqueness] Suppose that (H1)-(H6) hold and $\xi_0\in\mathcal{L}_{{\mathscr{F}}_0}^2$. There exists  $\delta>0$ depending only on $(L,T)$ such that FBSDE \eqref{fbsde1} has at most one solution satisfying \eqref{space} when $L_mC_f^{-1}\le\delta$.
	\end{theorem}
	
	\begin{proof}
		Let $\{(\hat{X}_t^i,\hat{p}_t^i,\hat{q}_t^i,\hat{\tilde{q}}_t^i),0\le t\le T\}$ solve FBSDE \eqref{fbsde1} such that \eqref{space} holds with initial $\xi_0\in\mathcal{L}_{\mathscr{F}_0}^2$, for $i=1,2$. We set
		\begin{equation*}
			m^i_t=\mathcal{L}(\hat{X}^i_t|\tilde{\mathscr{F}}_t),\quad i=1,2.
		\end{equation*}
		Recall that $\hat{u}^1_t$ and $\hat{u}^2_t$ satisfy
		\begin{equation}\label{uni1'}
			b_2(t)\hat{p}_t^i+\sigma_2(t)\hat{q}_t^i+\tilde{\sigma}_2(t)\hat{\tilde{q}}_t^i=-f_u(t,\hat{X}^i_t,\hat{u}_t^i),\quad t\in[0,T],\quad i=1,2.
		\end{equation}
		Let $\Delta\hat{X}_t=\hat{X}^2_t-\hat{X}^1_t$ and $\Delta\hat{p}_t=\hat{p}^2_t-\hat{p}^1_t$ for $t\in[0,T]$. By using Itô's lemma for $\Delta\hat{p}_t\Delta\hat{X}_t$, we get
		\begin{equation}\label{uni1}
			\begin{split}
				&\e[\Delta\hat{p}_T\Delta\hat{X}_T]-\Delta\hat{p}_0\Delta\hat{X}_0\\
				&=\e\big[\int_0^T (b_0(t,\hat{m}_t^2)-b_0(t,\hat{m}_t^1))\Delta\hat{p}_t+(\sigma_0(t,\hat{m}_t^2)-\sigma_0(t,\hat{m}_t^1))\Delta\hat{q}_t\\
				&\qquad\quad+(\tilde{\sigma}_0(t,\hat{m}_t^2)-\tilde{\sigma}_0(t,\hat{m}_t^1))\Delta\hat{\tilde{q}}_t+(b_2(t)\Delta\hat{p}_t+\sigma_2(t)\Delta\hat{q}_t+\tilde{\sigma}_2(t)\Delta\hat{\tilde{q}}_t)\Delta\hat{u}_t\\
				&\qquad\quad-\big(f_{0x}(t,\hat{X}_t^2,\hat{u}_t^2)-f_{0x}(t,\hat{X}_t^1,\hat{u}_t^1)\big)\Delta\hat{X}_t-\big(f_{1x}(t,\hat{X}_t^2,\hat{m}_t^2)-f_{1x}(t,\hat{X}_t^1,\hat{m}_t^1)\big)\Delta\hat{X}_t dt\big].  
			\end{split}
		\end{equation}
		We know from \eqref{uni1'} that 
		\begin{equation}\label{uni2}
			\begin{split}
				&(b_2(t)\Delta\hat{p}_t+\sigma_2(t)\Delta\hat{q}_t+\tilde{\sigma}_2(t)\Delta\hat{\tilde{q}}_t)\Delta\hat{u}_t-\big(f_{0x}(t,\hat{X}_t^2,\hat{u}_t^2)-f_{0x}(t,\hat{X}_t^1,\hat{u}_t^1)\big)\Delta\hat{X}_t\\
				&=-[(f_{0x},f_{0u})(t,\hat{X}_t^2,\hat{u}_t^2)-(f_{0x},f_{0u})(t,\hat{X}_t^1,\hat{u}_t^1)]\cdot (	\Delta \hat{X}_t,	\Delta \hat{u}_t),\quad t\in[0,T].
			\end{split}
		\end{equation}
		From the strict convexity of $f_0$ as assumed in (H4), we have that for $t\in[0,T]$,
		\begin{equation}\label{mp3}
			\begin{aligned}
				&f_0(t,\hat{X}_t^2,\hat{u}_t^2)-f_0(t,\hat{X}^1_t,\hat{u}^1_t)-[f_{0x}(t,\hat{X}^1_t,\hat{u}^1_t)\Delta \hat{X}_t+f_{0u}(t,\hat{X}^1_t,\hat{u}^1_t)\Delta \hat{u}_t]\geq C_f\Delta \hat{u}_t|^2;\\
				&f_0(t,\hat{X}^1_t,\hat{u}^1_t)-f_0(t,\hat{X}_t^2,\hat{u}_t^2)+[f_{0x}(t,\hat{X}_t^2,\hat{u}_t^2)\Delta \hat{X}_t+f_{0u}(t,\hat{X}_t^2,\hat{u}_t^2)\Delta \hat{u}_t]\geq C_f|\Delta u_t|^2.
			\end{aligned}
		\end{equation}
		From \eqref{mp3}, we have
		\begin{equation}\label{uni5}
			[(f_{0x},f_{0u})(t,\hat{X}_t^2,\hat{u}_t^2)-(f_{0x},f_{0u})(t,\hat{X}_t^1,\hat{u}_t^1)]\cdot (	\Delta \hat{X}_t,	\Delta \hat{u}_t)\geq 2C_f|\Delta\hat{u}_t|^2,\quad t\in[0,T].
		\end{equation}
		From the weak monotonicity assumption (H6), we know that
		\begin{equation}\label{uni7}
			\begin{split}
				&\e[\Delta\hat{p}_T\Delta\hat{X}_T]=\e[\big(g_x(\hat{X}_T^2,\hat{m}_T^2)-g_x(\hat{X}_T^1,\hat{m}_T^1)\big)\Delta\hat{X}_T]\geq 0;\\
				&\e[\big(f_{1x}(t,\hat{X}_t^2,\hat{m}_t^2)-f_{1x}(t,\hat{X}_t^1,\hat{m}_t^1)\big)\Delta\hat{X}_t]\geq 0,\quad t\in[0,T].
			\end{split}
		\end{equation}
		Plugging \eqref{uni2}, \eqref{uni5} and  \eqref{uni7} into \eqref{uni1}, using the Lipschitz continuity assumption (H5) and the average inequality, we have
		\begin{align}
			2C_f\e\big[\int_0^T|\Delta\hat{u}_t|^2\big]&\le\e\big[\int_0^T (b_0(t,\hat{m}_t^2)-b_0(t,\hat{m}_t^1))\Delta\hat{p}_t+(\sigma_0(t,\hat{m}_t^2)-\sigma_0(t,\hat{m}_t^1))\Delta\hat{q}_t\notag\\
			&\qquad\qquad+(\tilde{\sigma}_0(t,\hat{m}_t^2)-\tilde{\sigma}_0(t,\hat{m}_t^1))\Delta\hat{\tilde{q}}_tdt\big]\notag\\
			&\le \frac{3L_mT}{2}\e[\sup_{0\le t\le T}|\Delta\hat{X}_t|^2]+\frac{L_m}{2}\e\big[\int_0^T |\Delta\hat{p}_t|^2+|\Delta\hat{q}_t|^2+|\Delta\hat{\tilde{q_t}}|^2 dt\big],\label{uni7.5}
		\end{align}
		where we have used the following estimates
		\begin{equation*}
			\e[W_2(\hat{m}_t^1,\hat{m}_t^2)^2]\le \e\big[\e[|\Delta \hat{X}_t|^2|\tilde{\f}_t]\big]=\e[|\Delta \hat{X}_t|^2]\le \e[\sup_{0\le t\le T}|\Delta \hat{X}_t|^2],\quad t\in[0,T].
		\end{equation*}
		By standard estimates for SDEs and BSDEs, there exist two constants $C_1>0$ and $C_2>0$ depending only on $(L,T)$, such that
		\begin{align}
			&\e\big[\sup_{0\le t\le T}|\Delta\hat{X}_t|^2\big]\le C_1\e\big[\int_0^T|\Delta\hat{u}_t|^2dt\big];\label{uni8}\\
			&\e\big[\sup_{0\le t\le T}|\Delta\hat{p}_t|^2+\int_0^T|\Delta\hat{q}_t|^2+|\Delta\hat{\tilde{q}}_t|^2dt\big]  \le C_2\e\big[\sup_{0\le t\le T}|\Delta\hat{X}_t|^2+\int_0^T|\Delta\hat{u}_t|^2dt\big].\label{uni9}
		\end{align}
		From  \eqref{uni8}, \eqref{uni9} and \eqref{uni7.5}, we have
		\begin{equation*}
			4C_f\e\big[\int_0^T|\Delta\hat{u}_t|^2\big]\le L_m(C_2(1+C_1)(T+1)+3TC_1)\e\big[\int_0^T|\Delta\hat{u}_t|^2dt\big].
		\end{equation*}
		The constant 
		\begin{equation*}
			\delta:=2(C_2(1+C_1)(T+1)+3TC_1)^{-1}
		\end{equation*}
		depends only on $(L,T)$. If $L_mC_f^{-1}\le\delta$, then $\hat{u}^1=\hat{u}^2$ in $\lr_{\f}^2(0,T)$.
	\end{proof}
	
	\subsection{Existence}
	Next, we prove the existence result of the solution of FBSDE \eqref{fbsde1}. The idea is to show the existence result on a small time interval $[\tau,T]$ firstly and then extend the local solution to the whole time interval $[0,T]$. In this subsection, we always suppose that assumptions (H1)-(H6) hold. 
	
	The lemma below is an immediate consequence of \cite[Theorem 3]{SAWR}. Similar results can be found in \cite[Theorem 6.7]{NM} and \cite[Theorem 1.1]{AS}.
	\begin{lemma}\label{lem_m_x}
		Let $0\le s<\tau\le T$. Soppose that assumptions (H1), (H2) and (H5) hold. Given $\hat{u}\in \lr_{\f}^2(s,\tau)$ and $\eta\in\lr^2_{\f_s}$, there exists a unique solution $\{\hat{X}_t,s\le t\le \tau\}\in\sr^2_{\f}(s,\tau)$ to SDE:
		\begin{equation}\label{exix_hat}
			\hat{X}_t=\eta+\int_s^t b(r,\hat{X}_r,\hat{u}_r,\hat{m}_r) dr+\int_s^t \sigma(r,\hat{X}_r,\hat{u}_r,\hat{m}_r) dW_r+\int_s^t \tilde{\sigma}(r,\hat{X}_r,\hat{u}_r,\hat{m}_r) d\tilde{W}_r,\quad t\in[s,\tau],
		\end{equation}
		where $\hat{m}_t=\lr(\hat{X}_t|\tilde{\f}_t)$ for $t\in[s,\tau]$.
	\end{lemma}
	
	We now construct a map so that it has a fixed point as a solution to FBSDE \eqref{fbsde1}. Let $0\le s<\tau\le T$, $\eta\in\lr_{\f_s}^2$ and $v:\br\times\lr^2_{\f_\tau}\times\Omega\to\br$ satisfies, for $\mathbb{P}$-a.s.,
	\begin{align}
		&(v(x',\xi,\omega)-v(x,\xi,\omega))(x'-x)\geq 0,\quad x,x'\in\br,\quad \xi,\xi'\in\mathcal{L}_{\mathscr{F}_\tau}^2;\label{exiv1}\\
		&|v(x',\xi',\omega)-v(x,\xi,\omega)|^2\le C_v\big[|x'-x|^2+\e[(\xi'-\xi)^2|\tilde{\mathscr{F}}_\tau](\omega)\big],\quad x,x'\in\br,\quad \xi,\xi'\in\mathcal{L}_{\mathscr{F}_\tau}^2,\label{exiv2}
	\end{align}
	and 
	\begin{equation}\label{exiv0}
		\e[|v(0,0,\cdot)|^2]<+\infty.
	\end{equation}
	We define $\Phi^{s,\tau,\eta,v}: \lr_{\f}^2(s,\tau)\to\lr_{\f}^2(s,\tau)$ as follows. Given $\hat{u}\in \lr_{\f}^2(s,\tau)$, Lemma~\ref{lem_m_x} shows that SDE \eqref{exix_hat} has a unique solution $\hat{X}=\{\hat{X}_t,s\le t\le \tau\}\in\sr^2_{\f}(s,\tau)$ with initial value $\eta\in\lr^2_{\f_s}$. We set $\hat{m}_t=\lr(\hat{X}_t|\tilde{\f}_t)$ for $t\in[s,\tau]$ and denote by $\{(X_t,p_t,q_t,\tilde{q}_t),s\le t\le T\}$ the solution of the following FBSDE
	\begin{equation}\label{exix}
		\left\{
		\begin{aligned}
			&dX_t=b(t,X_t,u_t,\hat{m}_t)dt+{\sigma}(t,X_t,u_t,\hat{m}_t)dW_t+\tilde{\sigma}(t,X_t,u_t,\hat{m}_t)d\tilde{W}_t,\quad t\in(s,\tau];\\
			&dp_t=-\partial_x H(t,X_t,p_t,q_t,\tilde{q}_t,u_t,\hat{m}_t)dt+q_tdW_t+\tilde{q}_td\tilde{W}_t,\quad t\in[s,\tau);\\
			&X_s=\eta,\quad p_\tau=v(X_\tau,\hat{X}_\tau),
		\end{aligned}
		\right.
	\end{equation}
	with the optimality condition 
	\begin{equation*}
		u_t=\hat{u}(t,X_t,p_t,q_t,\tilde{q}_t),\quad t\in[s,\tau],
	\end{equation*}
	or equivalently,
	\begin{equation*}
		b_2(t)p_t+\sigma_2(t)q_t+\tilde{\sigma}_2(t)\tilde{q}_t+f_{0u}(t,X_t,u_t)=0,\quad t\in[s,\tau],
	\end{equation*}
	such that 
	\begin{equation}\label{define_estimate}
		\begin{split}
			&\e\big[\sup_{s\le t\le \tau}|(X_t,p_t)|^2+\int_s^{\tau}|(q_t,\tilde{q}_t)|^2)dt]\big]<+\infty. 
		\end{split}
	\end{equation}
	We set 
	\begin{equation*}
		\Phi^{s,\tau,\eta, v}(\hat{u}):=u=(u_t)_{s\le t\le\tau}.
	\end{equation*}
	Conditions \eqref{exiv1}-\eqref{exiv0} and Theorem~\ref{thm:mp'} ensure that $u$ is uniquely defined. Moreover, both inequalities \eqref{define_estimate} and \eqref{bound} and Lemma~\ref{lemma:u} yield that $u\in\lr_{\f}^2(s,\tau)$. Thus, $\Phi^{s,\tau,\eta, v}:\hat{u}\mapsto u$ maps $\lr_{\f}^2(s,\tau)$ into itself. Furthermore, the fixed point of $\Phi^{0,T,\xi_0,g_x}$ is the solution of FBSDE \eqref{fbsde1}. The lemma below gives a solution on a small time interval.
	
	\begin{lemma}\label{exi_lem1}
		Let $0<\tau\le T$ and $v:\br \times\lr^2_{\f_\tau}\times \Omega\to\br$ satisfies \eqref{exiv1}-\eqref{exiv0}. Suppose that assumptions (H1)-(H5) hold. 
		If 
		\begin{equation*}
			B_uC_f^{-1}\le (24LC_v)^{-1},
		\end{equation*}
		then, there exists $\gamma>0$ depending on $(L,T,C_f,C_v)$ such that, for any non-negative $s\in[\tau-\gamma,\tau)$ and $\eta\in\lr_{\f_s}^2$, there exists $\hat{u}^{s,\tau,\eta,v}\in \lr_{\f}^2(s,\tau)$ such that
		\begin{equation*}
			\Phi^{s,\tau,\eta,v}(\hat{u}^{s,\tau,\eta,v})=\hat{u}^{s,\tau,\eta,v}.
		\end{equation*}
	\end{lemma}
	
	The proof is given in Appendix~\ref{pf2.2}. Assumptions (H2)-(H5) ensure that $g_x$ satisfies conditions \eqref{exiv1}-\eqref{exiv0}. Suppose that 
	\begin{equation}\label{contract1}
		B_uC_f^{-1}\le (24L^2)^{-1}.
	\end{equation}
	Then, there exists $\gamma>0$ depending on $(L,T,C_f)$ such that, for any non-negative $t\in[T-\gamma,T)$ and $\eta\in\lr_{\f_t}^2$, there exists $\hat{u}^{t,T,\eta,g_x}\in\lr_{\f}^2(t,T)$ such that
	\begin{equation}\label{set:1}
		\Phi^{t,T,\eta,g_x}(\hat{u}^{t,T,\eta,g_x})=\hat{u}^{t,T,\eta,g_x}.
	\end{equation}
	Let $\Gamma$ be the set of $t\in [0,T]$ such that the following statement holds: for any $\eta\in\lr^2_{\f_t}$, there exists $\hat{u}^{t,T,\eta,g_x}\in\lr_{\f}^2(t,T)$ satisfying \eqref{set:1}. From the above we know that $T-\gamma\in \Gamma$. If $0\in \Gamma$, then we have completed the proof of the existence of the solution of FBSDE \eqref{fbsde1}. Suppose not, let $t_0=\inf \Gamma$ ($t_0$ can still be zero), $\gamma_0>0$ be sufficiently small so that $t_0+\gamma_0<T-\frac{\gamma}{2}$, and $\tau\in[t_0,t_0+\gamma_0)\cap\Gamma$. 
	
	Since $\tau\in\Gamma$, for any initial $\eta\in\lr^2_{\f_\tau}$, there exists a fixed point $\hat{u}=\{\hat{u}_t^{\tau,T,\eta,g_x},\tau\le t\le T\}\in \lr^2_{\f}(\tau,T)$ of the map $\Phi^{\tau,T,\eta,g_x}$.
	We denote by
	\begin{equation*}
		\{(\hat{X}^{\tau,T,\eta,g_x}_t,\hat{p}^{\tau,T,\eta,g_x}_t,\hat{q}^{\tau,T,\eta,g_x}_t,\hat{\tilde{q}}^{\tau,T,\eta,g_x}_t),\tau\le t\le T\}
	\end{equation*}
	the solution of the following FBSDE
	\begin{equation}\label{lem5.5_}
		\left\{
		\begin{aligned}
			&d\hat{X}^{\tau,T,\eta,g_x}_t=b(t,\hat{X}^{\tau,T,\eta,g_x}_t,\hat{u}_t^{\tau,T,\eta,g_x},\hat{m}_t^{\tau,T,\eta,g_x})dt+{\sigma}(t,\hat{X}^{\tau,T,\eta,g_x}_t,\hat{u}_t^{\tau,T,\eta,g_x},\hat{m}_t^{\tau,T,\eta,g_x})dW_t\\
			&\qquad\qquad\quad+\tilde{\sigma}(t,\hat{X}^{\tau,T,\eta,g_x}_t,\hat{u}_t^{\tau,T,\eta,g_x},\hat{m}_t^{\tau,T,\eta,g_x})d\tilde{W}_t,\quad t\in(s,\tau];\\
			&d\hat{p}^{\tau,T,\eta,g_x}_t=-\partial_x H(t,\hat{X}^{\tau,T,\eta,g_x}_t,\hat{p}^{\tau,T,\eta,g_x}_t,\hat{q}^{\tau,T,\eta,g_x}_t,\hat{\tilde{q}}^{\tau,T,\eta,g_x}_t,\hat{u}_t^{\tau,T,\eta,g_x},\hat{m}_t^{\tau,T,\eta,g_x})dt\\
			&\qquad\qquad\quad+\hat{q}^{\tau,T,\eta,g_x}_tdW_t+\hat{\tilde{q}}^{\tau,T,\eta,g_x}_td\tilde{W}_t,\quad t\in[s,\tau);\\
			&\hat{X}^{\tau,T,\eta,g_x}_\tau=\eta,\quad \hat{p}^{\tau,T,\eta,g_x}_T=g_x(\hat{X}^{\tau,T,\eta,g_x}_T,\hat{m}_T^{\tau,T,\eta,g_x}),
		\end{aligned}
		\right.
	\end{equation}
	where 
	\begin{equation*}
		\hat{m}_t^{\tau,T,\eta,g_x}=\lr(\hat{X}^{\tau,T,\eta,g_x}_t|\tilde{\f}_t),\quad \tau\le t\le T,
	\end{equation*}
	and with the optimality condition 
	\begin{equation*}
		b_2(t)\hat{p}_t^{\tau,T,\eta,g_x}+\sigma_2(t)\hat{q}_t^{\tau,T,\eta,g_x}+\tilde{\sigma}_2(t)\hat{\tilde{q}}_t^{\tau,T,\eta,g_x}+f_{0u}(t,\hat{X}_t^{\tau,T,\eta,g_x},\hat{u}_t^{\tau,T,\eta,g_x})=0,\quad t\in[s,\tau],
	\end{equation*}
	such that 
	\begin{equation}\label{lem5.5}
		\begin{split}
			&\e\big[\sup_{\tau\le t\le T}|(\hat{X}_t^{\tau,T,\eta,g_x},\hat{p}_t^{\tau,T,\eta,g_x})|^2+\int_\tau^{T}|(\hat{q}_t^{\tau,T,\eta,g_x},\hat{\tilde{q}}_t^{\tau,T,\eta,g_x})|^2)dt]\big]<+\infty. 
		\end{split}
	\end{equation}
	We then denote by
	\begin{align*}
		(X^{\tau,x,\eta},p^{\tau,x,\eta},q^{\tau,x,\eta},\tilde{q}^{\tau,x,\eta})&=\{(X_t^{\tau,x,\eta},p_t^{\tau,x,\eta},q_t^{\tau,x,\eta},\tilde{q}_t^{\tau,x,\eta}),\tau\le t\le T\}\\
		&\in\sr_{\f}^2(\tau,T)\times\sr_{\f}^2(\tau,T)\times\lr_{\f}^2(\tau,T)\times\lr_{\f}^2(\tau,T)
	\end{align*}
	the solution of the following FBSDE
	\begin{equation}\label{exi_lem2_fbsde}
		\left\{
		\begin{aligned}
			&dX_t=b(t,X_t,u_t,\hat{m}_t^{\tau,T,\eta,g_x})dt+{\sigma}(t,X_t,u_t,\hat{m}_t^{\tau,T,\eta,g_x})dW_t\\
			&\qquad\quad+\tilde{\sigma}(t,X_t,u_t,\hat{m}_t^{\tau,T,\eta,g_x})d\tilde{W}_t,\quad t\in(\tau,T];\\
			&dp_t=-\partial_x H(t,X_t,p_t,q_t,\tilde{q}_t,u_t,\hat{m}_t^{\tau,T,\eta,g_x})dt+q_tdW_t+\tilde{q}_td\tilde{W}_t,\quad t\in[\tau,T);\\
			&X_\tau=x,\quad p_T=g_x(X_T,\hat{m}_T^{\tau,T,\eta,g_x}),
		\end{aligned}
		\right.
	\end{equation}
	with the optimality condition 
	\begin{equation}\label{optimal_condition2}
		b_2(t)p_t+\sigma_2(t)q_t+\tilde{\sigma}_2(t)\tilde{q}_t+f_{0u}(t,X_t,u_t)=0,\quad t\in[\tau,T].
	\end{equation}
	We define $v:\br\times\lr^2_{\f_\tau}\times\Omega\to\br$ as
	\begin{equation*}
		v(x,\eta,\omega)=p_\tau^{\tau,x,\eta}(\omega),\quad x\in\br,\quad \eta\in\lr^2_{\f_\tau}.
	\end{equation*}
	From Theorem~\ref{thm:mp'}, $\{(X_t^{\tau,x,\eta},p_t^{\tau,x,\eta},q_t^{\tau,x,\eta},\tilde{q}_t^{\tau,x,\eta}),\tau\le t\le T\}$ is uniquely defined, and thus $v$ is well-defiend. We have the following estimates of $v$, whose proof is given in Appendix~\ref{pf2.3}.
	
	\begin{lemma}\label{exi_lem2}
		Let assumptions (H1)-(H6) and inequality \eqref{contract1} be satisfied. Then, there exists $\delta>0$ depending only on $(L,T)$ such that $v$ defined above satisfies, $\mathbb{P}$-a.s.
		\begin{align}
			&(v(x',\eta)-v(x,\eta))(x'-x)\geq 0,\quad x,x'\in\br,\quad \eta,\eta'\in\lr^2_{\f_\tau};\label{exi_lem2_v1}\\
			&|v(x',\eta')-v(x,\eta)|^2\le C_v\big[|x'-x|^2+\e[|\eta'-\eta|^2|\tilde{\f}_\tau]\big],\quad x,x'\in\br,\quad \eta,\eta'\in\lr^2_{\f_\tau},\label{exi_lem2_v2}
		\end{align}
		when $L_mC_f^{-1}\le\delta$, with the constant $C_v=C_{L,T}(1+\frac{1}{C_f})^4$, where $C_{L,T}$ is a constant depending only on $(L,T)$.
	\end{lemma}
	We now attempt to extend the solution further. If we suppose that 
	\begin{equation*}\label{constract2}
		B_uC_f^{-1}(1+C_f^{-1})^{4}\le \min\{(24L^2)^{-1},(24LC_{L,T})^{-1}\},
	\end{equation*}
	then we easily get
	\begin{equation*}
		B_uC_f^{-1}\le \min\{(24L^2)^{-1},(24LC_v)^{-1}\}.
	\end{equation*}
	When $L_mC_f^{-1}$ is small enough, Lemma~\ref{exi_lem2} and \eqref{lem5.5} ensure that \eqref{exiv1}-\eqref{exiv0} hold. Let $s\in[0,\tau)$ to be determined and $\eta\in\lr_{\f_s}^2$. Consider the map $\Phi^{s,\tau,\eta,v}$ as defined above. From Lemma~\ref{exi_lem1}, there exist a constant $\gamma'>0$ depending only on $(L,T,C_f)$ such that, for any non-negative $s\in[\tau-\gamma',\tau)$, there exist $\hat{u}^{s,\tau,\eta,v}\in \lr_{\f}^2(s,\tau)$ such that
	\begin{equation*}
		\Phi^{s,\tau,\eta,v}(\hat{u}^{s,\tau,\eta,v})=\hat{u}^{s,\tau,\eta,v}.
	\end{equation*}
	We denote by $\{\hat{X}_t^{s,\tau,\eta,v},s\le t\le \tau\}\in\sr^2_{\f}(s,\tau)$ the state process  corresponding to $\hat{u}^{s,\tau,\eta,v}$. We construct the following control $\hat{u}^{s,T,\eta,g_x}$ by letting
	\begin{equation*}
		\hat{u}_t^{s,T,\eta,g_x}=
		\left\{
		\begin{array}{ll}
			\hat{u}_t^{s,\tau,\eta,v},\quad &if \quad s\le t<\tau;  \\
			\hat{u}_t^{\tau,T,\hat{X}_\tau^{s,\tau,\eta,v},g_x}, \quad &if \quad \tau\le t\le T.
		\end{array}
		\right.
	\end{equation*}
	Since $\hat{u}^{s,\tau,\eta,v}\in \lr^2_{\f}(s,\tau)$ and $\hat{u}^{\tau,T,\hat{X}_\tau^{s,\tau,\eta,v},g_x}\in \lr^2_{\f}(\tau,T)$, we have $\hat{u}^{s,T,\eta,g_x}\in\lr^2_{\f}(s,T)$. From the definition of $\Phi$, we have
	\begin{equation*}
		\Phi^{s,T,\eta,g_x}(\hat{u}^{s,T,\eta,g_x})=\hat{u}^{s,T,\eta,g_x}.
	\end{equation*}
	
	We have shown that $s\in \Gamma$. From Lemmas~\ref{exi_lem1} and \ref{exi_lem2}, we know that $\gamma'$ depends only on $(L,T,C_f)$, which is independent of $\tau$. Therefore, we can select $\gamma_0$ sufficiently small so that $\gamma_0<\gamma'$, so that $\tau-\gamma'\le t_0+\gamma_0-\gamma'<t_0$. Thus, we can select $s$ to be strictly less than $t_0$ or $s=0$ to give a contradiction to our assumption that $0\notin \Gamma$ and $t_0=\inf\Gamma$. Therefore $0\in \Gamma$, and we have the following existence theorem:
	
	\begin{theorem}[Existence] Suppose that assumptions (H1)-(H6) hold and $\xi_0\in\mathcal{L}_{\mathscr{F}_0}^2$. Then, there exists $\delta>0$ depending only on $(L,T)$, such that FBSDE \eqref{fbsde1} has a solution satisfying \eqref{space} when
		\begin{equation*}
			\max\{B_uC_f^{-1}(1+C_f^{-1})^4, L_mC_f^{-1}\}\le\delta.
		\end{equation*}
		
	\end{theorem}
	%%%%%%%%%%%%%%%%%%%%%%%%%%%%%%%%%%%%%%%%%%%%%%%%%%%%%%%%%%%%%%%%%%%%%%%%%%%%
	
	%%%%%%%%%%%%%%%%%%%%%%%%%%%%%%%%%%%%%%%%%%%%%%%%%%%%%%%%%%%%%%%%%%%%
	\normalsize 
	\appendix
	\appendixpage
	\addappheadtotoc
	
	\section{Proof of Lemma~\ref{main1_lem1}.}\label{proof1}
	We set $\Delta X_t=X_t^2-X_t^1$ . The differences $\Delta u_t$, $\Delta p_t$, $\Delta q_t$ and $\Delta \tilde{q}_t$ are defined in a similar way. Then, $\{(\Delta X_t,\Delta p_t,\Delta q_t,\Delta \tilde{q}_t),0\le t\le T\}$ satisfy the following FBSDE
	\begin{equation*}
		\left\{
		\begin{aligned}
			&d\Delta X_t=[\gamma((b_0(t,m_t^2)-b_0(t,m_t^1))+b_1(t)\Delta X_t+b_2(t)\Delta u_t)+\Delta\ii_t^b]dt\\
			&\quad\quad\quad+[\gamma((\sigma_0(t,m_t^2)-\sigma_0(t,m_t^1))+\sigma_1(t)\Delta X_t+\sigma_2(t)\Delta u_t)+\Delta\ii_t^{\sigma}]dW_t\\
			&\quad\quad\quad+[\gamma((\tilde{\sigma}_0(t,m_t^2)-\tilde{\sigma}_0(t,m_t^1))+\tilde{\sigma}_1(t)\Delta X_t+\tilde{\sigma}_2(t)\Delta u_t)+\Delta\ii_t^{\tilde{\sigma}}]d\tilde{W}_t,\quad t\in(0,T];\\
			&d\Delta p_t=-[\gamma(b_1(t)\Delta p_t+\sigma_1(t)\Delta q_t+\tilde{\sigma}_t\Delta\tilde{q}_t+(f_x(t,X_t^2,u_t^2,m_t^2)-f_x(t,X_t^1,u_t^1,m_t^1)))\\
			&\quad\qquad\quad+\Delta\ii_t^f]dt+\Delta q_tdW_t+\Delta\tilde{q}_td\tilde{W}_t,\quad t\in[0,T);\\
			&\Delta X_0=\Delta\xi,\quad \Delta p_T=\gamma(g_x(X_T^2,m_T^2)-g_x(X_T^1,m_T^1))+\Delta\ii_T^g,
		\end{aligned}
		\right.
	\end{equation*}
	with the condition
	\begin{equation}\label{main1_lem1_0.2}
		b_2(t)\Delta p_t+\sigma_2(t)\Delta q_t+\tilde{\sigma}(t)\Delta \tilde{q}_t+(f_{0u}(t,X_t^2,u_t^2)-f_{0u}(t,X_t^1,u_t^1))=0,\quad t\in[0,T].
	\end{equation}
	First we note the fact that
	\begin{equation}\label{main1_lem1_0.1}
		\e[W_2(m_t^1,m_t^2)^2]\le \e[\e[|\Delta X_t|^2|\tilde{\f}_t]]=\e[|\Delta X_t|^2]\le \e[\sup_{0\le t\le T}|\Delta X_t|^2],\quad t\in[0,T].
	\end{equation}
	Under assumptions (H1), (H3), (H5) and the estimates \eqref{main1_lem1_0.1}, by standard estimates for SDEs and BSDEs, there exist two constants $C_1>0$ and $C_2>0$ depending only on $(L,T)$, such that 
	\begin{align}
		&\e[\sup_{0\le t\le T}|\Delta X_t|^2]\le C_1\e\big[|\Delta\xi|^2+\int_0^T \gamma|\Delta u_t|^2+|(\Delta\ii_t^b,\Delta\ii_t^{\sigma},\Delta\ii_t^{\tilde{\sigma}})|^2 dt\big];\label{main1_lem1_1}\\
		&\e\big[\sup_{0\le t\le T}|\Delta p_t|^2+\int_0^T |(\Delta q_t,\Delta \tilde{q}_t)|^2 dt\big]\notag\\
		&\qquad\qquad\quad \le C_2\e\big[\gamma\sup_{0\le t\le T}|\Delta X_t|^2+|\Delta\ii_T^g|^2+\int_0^T\gamma|\Delta u_t|^2+|\Delta\ii_t^f|^2dt\big].\label{main1_lem1_2}
	\end{align}
	Applying Itô's lemma on $\Delta p_t\Delta X_t$ and taking expection, we have
	\begin{equation}\label{main1_lem1_2.5}
		\begin{split}
			&\e[\Delta p_T\Delta X_T]-\e[\Delta p_0\Delta\xi]\\
			&=\e\big[\int_0^T \gamma[(b_0(t,m_t^2)-b_0(t,m_t^1))\Delta p_t+(\sigma_0(t,m_t^2)-\sigma_0(t,m_t^1))\Delta q_t\\
			&\qquad\qquad\quad+(\tilde{\sigma}_0(t,m_t^2)-\tilde{\sigma}_0(t,m_t^1))\Delta \tilde{q}_t]\\
			&\qquad\quad+\gamma[(b_2(t)\Delta p_t+\sigma_2(t)\Delta q_t+\tilde{\sigma}_2(t)\Delta\tilde{q})\Delta u_t\\
			&\qquad\qquad\quad-(f_x(t,X_t^2,u_t^2,m_t^2)-f_x(t,X_t^1,u_t^1,m_t^1))\Delta X_t]\\
			&\qquad\quad+(\Delta p_t\Delta\ii_t^b+\Delta q_t\Delta\ii_t^{\sigma}+\Delta\tilde{q}_t\Delta\ii_t^{\tilde{\sigma}}-\Delta X_t\Delta\ii_t^f) dt\big].
		\end{split}
	\end{equation}
	By using the Lipschitz-continuity assumption (H5) and the average inequality, we have
	\begin{align}
		&\e\big[\int_0^T\gamma[(b_0(t,m_t^2)-b_0(t,m_t^1))\Delta p_t+(\sigma_0(t,m_t^2)-\sigma_0(t,m_t^1))\Delta q_t\notag\\
		&\qquad\qquad+(\tilde{\sigma}_0(t,m_t^2)-\tilde{\sigma}_0(t,m_t^1))\Delta \tilde{q}_t]dt\big]\notag\\
		&\le \e\big[\int_0^T\gamma L_m W_2(m_t^1,m_t^2)(|\Delta p_t|+|\Delta q_t|+|\Delta \tilde{q}_t|)dt\big]\notag\\
		&\le \frac{3TL_m\gamma}{2}\e[\sup_{0\le t\le T}|\Delta X_t|^2]+\frac{TL_m\gamma}{2}\e[\sup_{0\le t\le T}|\Delta p_t|^2]+\frac{L_m\gamma}{2}\e\big[\int_0^T|\Delta q_t|^2+|\Delta\tilde{q}_t|^2dt\big].\label{main1_lem1_8}
	\end{align}
	From \eqref{main1_lem1_0.2} and the convex assumption (H4), we have
	\begin{equation}\label{main1_lem1_3}
		\begin{split}
			&(b_2(t)\Delta p_t+\sigma_2(t)\Delta q_t+\tilde{\sigma}(t)\Delta \tilde{q}_t)\Delta u_t-(f_{0x}(t,X_t^2,u_t^2)-f_{0x}(t,X_t^1,u_t^1))\Delta X_t\\
			&\le -2C_f|\Delta\hat{u}_t|^2,\quad t\in[0,T].
		\end{split}
	\end{equation}
	From the weak monotonicity assumption (H6) and the fact that $m_t^i=\lr(X_t^i|\tilde{\f}_t)$, we have
	\begin{equation}\label{main1_lem1_5}
		\begin{split}
			&\e[(f_{1x}(t,X_t^2,m_t^2)-f_{1x}(t,X_t^1,m_t^1))\Delta X_t]\\
			&\quad=\e[\e[(f_{1x}(t,X_t^2,m_t^2)-f_{1x}(t,X_t^1,m_t^1))\Delta X_t|\tilde{\f}_t]]\geq 0,\quad t\in[0,T];\\
			&\e[\Delta p_T\Delta X_T]=\e[\gamma(g_x(X_T^2,m_T^2)-g_x(X_T^1,m_T^1))\Delta X_T+\Delta X_T\Delta\ii_T^g]\geq\e[\Delta X_T\Delta\ii_T^g].
		\end{split}
	\end{equation}
	Plugging \eqref{main1_lem1_8}, \eqref{main1_lem1_3} and \eqref{main1_lem1_5} into \eqref{main1_lem1_2.5}, and using the average inequality, we have for any $\epsilon>0$, 
	\begin{equation*}
		\begin{split}
			&2C_f\gamma\e\big[\int_0^T |\Delta u_t|^2 dt\big]\\
			&\le \e\big[|\Delta p_0||\Delta\xi|+|\Delta X_T||\Delta\ii_T^g|+\int_0^T (\Delta p_t\Delta\ii_t^b+\Delta q_t\Delta\ii_t^{\sigma}+\Delta\tilde{q}_t\Delta\ii_t^{\tilde{\sigma}}-\Delta X_t\Delta\ii_t^f) dt\big]\\
			&\quad+\frac{3TL_m\gamma}{2}\e[\sup_{0\le t\le T}|\Delta X_t|^2]+\frac{TL_m\gamma}{2}\e[\sup_{0\le t\le T}|\Delta p_t|^2]+\frac{L_m\gamma}{2}\e\big[\int_0^T|\Delta q_t|^2+|\Delta\tilde{q}_t|^2dt\big]\\
			&\le (\frac{3TL_m\gamma}{2}+\epsilon)\e[\sup_{0\le t\le T}|\Delta X_t|^2]+(\frac{TL_m\gamma}{2}+\epsilon)\e[\sup_{0\le t\le T}|\Delta p_t|^2]\\
			&\quad+(\frac{L_m\gamma}{2}+\epsilon)\e\big[\int_0^T|\Delta q_t|^2+|\Delta\tilde{q}_t|^2dt\big]+C(T,\epsilon)(\e[|\Delta\xi|^2]+\|\Delta\ii\|^2_{\iii}).
		\end{split}
	\end{equation*}
	Here, the notation $C(T,\epsilon)$ stands for a constant depending only on $T$ and $\epsilon$. Plugging \eqref{main1_lem1_1} and \eqref{main1_lem1_2} into above, we have
	\begin{equation*}
		\begin{split}
			&2C_f\gamma\e\big[\int_0^T |\Delta u_t|^2 dt\big]\\
			&\le  C(L,T,\epsilon)(\e[|\Delta\xi|^2]+\|\Delta\ii\|^2_{\iii})+(C_1C_2+C_1+C_2)\gamma\epsilon\e\big[\int_0^T |\Delta u_t|^2 dt\big]\\
			&\qquad+\frac{3TC_1+\max\{1,T\}(C_1+1)C_2}{2}L_m\gamma\e\big[\int_0^T |\Delta u_t|^2 dt\big].
		\end{split}
	\end{equation*}
	The constant
	\begin{equation*}
		\delta:=2(3TC_1+\max\{1,T\}(C_1+1)C_2)^{-1}
	\end{equation*}
	depends only on $(L,T)$. If $L_mC_f^{-1}\le\delta$, then, we choose 
	\begin{equation*}
		\epsilon=\frac{C_f}{2(C_1C_2+C_1+C_2)}
	\end{equation*}
	to get  
	\begin{equation}\label{main1_lem1_9}
		\gamma\e\big[\int_0^T |\Delta u_t|^2 dt\big]\le C(L,T,C_f)(\e[|\Delta\xi|^2]+\|\Delta \ii\|^2_{\iii}).
	\end{equation}
	Plugging \eqref{main1_lem1_9} into \eqref{main1_lem1_1} and \eqref{main1_lem1_2} respectively, we have
	\begin{equation*}
		\begin{split}
			&\e\big[\sup_{0\le t\le T}|(\Delta X_t,\Delta p_t)|^2+\int_0^T|(\Delta q_t,\Delta\tilde{q}_t)|^2dt\big]\le C(L,T,C_f)(\e[|\Delta\xi|^2]+\|\Delta\ii\|^2_{\iii}).
		\end{split}
	\end{equation*}
	From Lemma~\ref{lemma:u} we know that
	\begin{equation*}
		|\Delta u_t|\le C(L,C_f)(|\Delta X_t|+|\Delta p_t|+|\Delta q_t|+|\Delta\tilde{q}_t|),\quad t\in[0,T].
	\end{equation*}
	So we eventually have
	\begin{equation*}
		\|\Theta^1-\Theta^2\|_{\mathbb{S}}^2\le C(L,T,C_f)([\e|\xi^1-\xi^2|^2]+\|\ii^1-\ii^2\|_{\iii}^2).
	\end{equation*}
	\endproof
	%%%%%%%%%%%%%%%%%%%%%%%%%%%%%%%%%%%%%%%%%%%%
	\section{Proof of Lemma~\ref{exi_lem1}.}\label{pf2.2}
	Let $\eta\in\lr_{\f_s}^2$ and $\hat{u}^1,\hat{u}^2\in \lr_{\f}^2(s,\tau)$. We denote by $\{\hat{X}_t^i,s\le t\le\tau\}$ the state process corresponding to $\hat{u}^i$ as in \eqref{exix_hat} and set $\hat{m}^i_t=\lr(\hat{X}^i_t|\tilde{\f}_t)$ for $t\in[s,\tau]$ and $i=1,2$. We then denote by $\{(X_t^i,p_t^i,q_t^i,\tilde{q}_t^i),s\le t\le\tau\}$ the solution of FBSDE \eqref{exix} corresponding to $\{\hat{m}^i_t,s\le t\le \tau\}$ for $i=1,2$. We set $\Delta X_t=X_t^2-X_t^1$ for $t\in[s,\tau]$. The differences $(\Delta p_t,\Delta q_t,\Delta \tilde{q}_t,\Delta u_t,\Delta \hat{u}_t,\Delta \hat{X}_t)$ are defined in a similar way. Then $\{(\Delta X_t,\Delta p_t,\Delta q_t,\Delta \tilde{q}_t),s\le t\le\tau\}$ satisfy the following FBSDE
	\begin{equation*}
		\left\{
		\begin{aligned}
			&d\Delta X_t=[(b_0(t,\hat{m}^2_t))-b_0(t,\hat{m}^1_t)+b_1(t)\Delta X_t+b_2(t)\Delta u_t]dt\\
			&\quad\quad\quad +[(\sigma_0(t,\hat{m}^2_t)-\sigma_0(t,\hat{m}^1_t))+\sigma_1(t)\Delta X_t+{\sigma}_2(t)\Delta u_t]dW_t\\
			&\quad\quad\quad +[(\tilde{\sigma}_0(t,\hat{m}^2_t)-\tilde{\sigma}_0(t,\hat{m}^1_t))+\tilde{\sigma}_1(t)\Delta X_t+\tilde{\sigma}_2(t)\Delta u_t] d\tilde{W}_t,\quad t\in(s,\tau];\\
			&d\Delta p_t=-\big[b_1(t)\Delta p_t+\sigma_1(t)\Delta q_t+\tilde{\sigma}_1(t)\Delta \tilde{q}_t+(f_{0x}(t,X^2_t,u^2_t)-f_{0x}(t,X^1_t,u^1_t))\\
			&\quad\quad\quad+(f_{1x}(t,X^2_t,\hat{m}_t^2)-f_{1x}(t,X^1_t,\hat{m}_t^1))\big]dt+\Delta q_tdW_t+\Delta\tilde{q}_td\tilde{W}_t,\quad t\in[s,\tau);\\
			&\Delta X_s=0, \quad \Delta p_\tau=v(X_\tau^2,\hat{X}^2_\tau)-v(X_\tau^1,\hat{X}^1_\tau).
		\end{aligned}
		\right.
	\end{equation*}
	From Lemma~\ref{lemma:u} we know that
	\begin{equation}\label{exi_lem1_1}
		|\Delta u_t|\le \frac{L}{2C_f}|\Delta X_t|+\frac{L}{2C_f}|\Delta p_t|+\frac{B_u}{2C_f}(|\Delta q_t|+|\Delta \tilde{q}_t|),\quad t\in[0,T],
	\end{equation}
	and recall that $B_u\le L$. From \eqref{exi_lem1_1} and assumptions (H1), (H3) and (H5), we have 
	\begin{align}
		&\e[\sup_{s\le t\le\tau}|\Delta X_t|^2]\notag\\
		&\le (\tau-s)C(L,T)\e\big[\sup_{s\le t\le\tau}|(\Delta \hat{X}_t,\Delta {X}_t)|^2+\int_s^{\tau}|\Delta u_t|^2dt\big]+6B_u^2\e\big[\int_s^{\tau}|\Delta u_t|^2dt\big]\notag\\
		&\le (\tau-s)C(L,T,C_f)\e\big[\sup_{s\le t\le\tau}|(\Delta \hat{X}_t,\Delta {X}_t,\Delta p_t)|^2+\int_s^{\tau}|(\Delta q_t,\Delta\tilde{q}_t)|^2dt\big]\notag\\
		&\quad +\frac{6B_u^2L^2}{C_f^2}\e\big[\int_s^{\tau}|(\Delta q_t,\Delta\tilde{q}_t)|^2dt\big].\label{exi_lem1_2}
	\end{align}
	Here, the notation $C(L,T,C_f)$ stands for a constant depending only on $L$, $T$ and $C_f$, and we have used the following estimates
	\begin{equation*}
		\e[W_2(\hat{m}_t^1,\hat{m}_t^2)^2]\le \e[\e[|\Delta \hat{X}_t|^2|\tilde{\f}_t]]=\e[|\Delta \hat{X}_t|^2]\le \e[\sup_{0\le t\le T}|\Delta \hat{X}_t|^2],\quad t\in[0,T].
	\end{equation*}
	Similarly, by using Doob's inequality, Cauchy's inequality and \eqref{exi_lem1_1}, we have
	\begin{align}
		&\e[\sup_{s\le t\le\tau}|\Delta p_t|^2]\notag\\
		%&=\e\bigg[\sup_{s\le t\le\tau}\bigg(\e\Big[\big(v(X_\tau^2,\hat{X}_\tau^2)-v(X_\tau^1,\hat{X}_\tau^1)\big)\\
		%&\quad\quad\quad\quad\quad+\int_t^\tau b_1(r)\Delta p_r+\sigma_1(r)\Delta q_r+(f_x(r,X^2_r,u^2_r,\hat{m}^2_r)-f_x(r,X^1_r,u^1_r,\hat{m}^1_r)) dr\Big|{\f}_t\Big]\bigg)^2\bigg]\\
		&\le 2\e[|v(X_\tau^2,\hat{X}_\tau^2)-v(X_\tau^1,\hat{X}_\tau^1)|^2]+8\e\big[\big(\int_s^{\tau}|b_1(r)\Delta p_r|+|\sigma_1(r)\Delta q_r|+|\tilde{\sigma}_1(r)\Delta \tilde{q}_r|\notag\\
		&\qquad\qquad\qquad\qquad\qquad\qquad\qquad\qquad\qquad +|f_{x}(r,X^2_r,u^2_r,\hat{m}^2_r)-f_{x}(r,X^1_r,u^1_r,\hat{m}^1_r)| dr\big)^2\big]\notag\\
		&\le (\tau-s)C(L,T,C_f)\e\big[\sup_{s\le t\le\tau}|(\Delta \hat{X}_t,\Delta {X}_t,\Delta p_t)|^2+\int_s^{\tau}|(\Delta q_t,\Delta\tilde{q}_t)|^2dt\big]\notag\\
		&\quad +4C_v^2\e[\sup_{s\le t\le\tau}|(\Delta \hat{X}_t,\Delta {X}_t)|^2].\label{exi_lem1_3}
	\end{align}
	We also have from \eqref{exi_lem1_1} that
	\begin{align}
		&\e\big[\int_s^\tau|(\Delta q_t,\Delta\tilde{q}_t)|^2dt\big]=\e\big[\big|\int_s^\tau\Delta{q}_tdW_t+\int_s^\tau\Delta\tilde{q}_td\tilde{W}_t\big|^2\big]\notag\\
		&=\e\big[\big|\Delta p_\tau-\Delta p_s+\int_s^\tau b_1(t)\Delta p_t+\sigma_1(t)\Delta q_t+\tilde{\sigma}_1(r)\Delta \tilde{q}_r\notag\\
		&\qquad\quad+(f_{x}(t,X^2_t,u^2_t,\hat{m}^2_t)-f_{x}(t,X^1_t,u^1_t,\hat{m}^1_t))dt \big|^2\big]\notag\\
		&\le (\tau-s)C(L,T,C_f)\e\big[\sup_{s\le t\le\tau}|(\Delta \hat{X}_t,\Delta {X}_t,\Delta p_t)|^2+\int_s^{\tau}|(\Delta q_t,\Delta\tilde{q}_t)|^2dt\big]\notag\\
		&\qquad +6\e[\sup_{s\le t\le\tau}|\Delta p_t|^2].\label{exi_lem1_4}
	\end{align}
	From \eqref{exi_lem1_2}, \eqref{exi_lem1_3} and \eqref{exi_lem1_4}, we deduce that
	\begin{equation}\label{exi_lem1_4'}
		\begin{split}
			&\e[\sup_{s\le t\le\tau}|\Delta X_t|^2]\\
			&\qquad\le \frac{\frac{144B_u^2L^2C_v^2}{C_f^2}+(\tau-s)C(L,T,C_f,C_v)}{1-(\tau-s)C(L,T,C_f,C_v)}\e[\sup_{s\le t\le\tau}|(\Delta X_t,\Delta \hat{X}_t)|^2];\\
			&\e[\sup_{s\le t\le\tau}|\Delta p_t|^2]\\
			&\qquad\le \frac{4C_v^2+(\tau-s)C(L,T,C_f,C_v)}{1-(\tau-s)C(L,T,C_f,C_v)}\e[\sup_{s\le t\le\tau}|(\Delta X_t,\Delta \hat{X}_t)|^2];\\
			&\e\big[\int_s^\tau|(\Delta q_t,\Delta\tilde{q}_t)|^2\big]\\
			&\qquad\le \frac{24C_v^2+(\tau-s)C(L,T,C_f,C_v)}{1-(\tau-s)C(L,T,C_f,C_v)}\e[\sup_{s\le t\le\tau}|(\Delta X_t,\Delta \hat{X}_t)|^2].
		\end{split}
	\end{equation}
	From the condition
	\begin{equation}\label{A8}
		\frac{B_u}{C_f}\le \frac{1}{24LC_v},
	\end{equation}
	we have that when $(\tau-s)$ is small enough, 
	\begin{equation}\label{exi_lem1_5'}
		\begin{split}
			&\e[\sup_{s\le t\le\tau}|\Delta X_t|^2]\le\frac{1}{2}\e[\sup_{s\le t\le\tau}|\Delta \hat{X}_t|^2].
		\end{split}
	\end{equation}
	From \eqref{exi_lem1_4'}, \eqref{exi_lem1_5'} and \eqref{exi_lem1_1}, we deduce that when $(\tau-s)$ is small enough,
	\begin{equation}\label{exi_lem1_5}
		\begin{split}
			\e\big[\int_s^{\tau}|\Delta u_t|^2dt\big]\le \frac{48L^2C_v^2}{C_f^2}\e[\sup_{s\le t\le\tau}|\Delta \hat{X}_t|^2].
		\end{split}
	\end{equation}
	
	Similar as the above, we have
	\begin{equation}\label{exi_lem1_6}
		\e[\sup_{s\le t\le\tau}|\Delta \hat{X}_t|^2]\le \frac{6B_u^2+(\tau-s)C(L,T,C_f,C_v)}{1-(\tau-s)C(L,T,C_f,C_v)}\e\Big[\int_s^{\tau}|\Delta \hat{u}_t|^2dt\Big].
	\end{equation}
	From \eqref{A8}, \eqref{exi_lem1_5} and \eqref{exi_lem1_6}, we deduce that

	\begin{equation*}
		\begin{split}
			\e\big[\int_s^\tau |\Delta u_t|^2 dt\big]
			&\le \frac{\frac{1}{2}+(\tau-s)C(L,T,C_f,C_v)}{1-(\tau-s)C(L,T,C_f,C_v)}\e\big[\int_s^{\tau}|\Delta \hat{u}_t|^2dt\big].
		\end{split}
	\end{equation*}
	It follows that when $(\tau-s)$ is small enough,
	\begin{equation*}
		||\Phi(\hat{u}^2)-\Phi(\hat{u}^1)||^2_{\lr_{\f}^2(s,\tau)}=\e\big[\int_s^\tau |\Delta u_t|^2 dt\big]\le \frac{3}{4}||\hat{u}^2-\hat{u}^1||^2_{\lr_{\f}^2(s,\tau)}.
	\end{equation*}
	As a result, we get a contraction map for sufficiently small $(\tau-s)$ depending only on $(L,T,C_v,C_f)$ as desired.
	\endproof
	%%%%%%%%%%%%%%%%%%%%%%%%%%%%%%%%%%%%%%%%%%%%
	\section{Proof of Lemma~\ref{exi_lem2}.}\label{pf2.3}
	In this section, we give the proof of \eqref{exi_lem2_v1} and \eqref{exi_lem2_v2}, respectively. From the condition \eqref{contract1} and Lemma~\ref{exi_lem1}, we know that $v$ is well-defined.
	
	\subsection{Proof of \eqref{exi_lem2_v1}.}
	Let $\eta\in\lr^2_{\f_\tau}$ and $x_1,x_2\in\br$. We denote by $\{(\hat{X}_t,\hat{p}_t,\hat{q}_t,\hat{\tilde{q}}_t),\tau\le t\le T\}$ the solution of FBSDE \eqref{lem5.5_} with initial condition $\hat{X}_\tau=\eta$, and set $\hat{m}_t=\lr(\hat{X}_t|\tilde{\f}_t)$ for $t\in[\tau,T]$. We then denote by $\{(X_t^i,p_t^i,q_t^i,\tilde{q}_t^i),\tau\le t\le T\}$ the solution of FBSDE \eqref{exi_lem2_fbsde} corresponding to $\{\hat{m}_t,\tau\le t\le T\}$ with initial conditions $X_\tau^i=x_i$ for $i=1,2$. We set $\Delta X_t=X_t^2-X_t^1$ and $\Delta p_t=p_t^2-p_t^1$ for $t\in[\tau,T]$. By applying Itô's lemma to $\Delta p_t\Delta X_t$ from $\tau$ to $T$ and taking expectation conditional on $\f_\tau$ (denoted by $\e_\tau[\cdot]$), we get
	\begin{equation}\label{C1_1}
		\begin{split}
			&\e_{\tau}[\Delta p_T\Delta X_T]-\Delta p_\tau\Delta x\\
			&=\e_\tau\big[\int_\tau^T (b_2(t)\Delta p_t+\sigma_2(t)\Delta q_t+\tilde{\sigma}_2(t)\Delta\tilde{q}_t)\Delta u_t\\
			&\qquad\qquad-\big(f_{0x}(t,{X}_t^2,{u}_t^2)-f_{0x}(t,{X}_t^1,{u}_t^1)\big)\Delta{X}_t\\
			&\qquad\qquad-\big(f_{1x}(t,{X}_t^2,\hat{m}_t)-f_{1x}(t,{X}_t^1,\hat{m}_t)\big)\Delta{X}_t dt\big].
		\end{split}
	\end{equation}
	From the optimal conditions \eqref{optimal_condition2} of  $u^1$ and $u^2$, we have
	\begin{equation*}
		b_2(t)\Delta p_t+\sigma_2(t)\Delta q_t+\tilde{\sigma}_2(t)\Delta\tilde{q}_t=-\big(f_{0u}(t,{X}_t^2,{u}_t^2)-f_{0u}(t,{X}_t^1,{u}_t^1)\big)\Delta{u}_t,\quad t\in[\tau,T].
	\end{equation*}
	From the convexity assumption (H4), we have
	\begin{equation}\label{C1_2}
		\begin{split}
			&(b_2(t)\Delta p_t+\sigma_2(t)\Delta q_t+\tilde{\sigma}_2(t)\Delta\tilde{q}_t)\Delta{u}_t-\big(f_{0x}(t,{X}_t^2,{u}_t^2)-f_{0x}(t,{X}_t^1,{u}_t^1)\big)\Delta{X}_t\\
			&\quad\quad =-\big((f_{0x},f_{0u})(t,{X}_t^2,{u}_t^2)-(f_{0x},f_{0u})(t,{X}_t^1,{u}_t^1)\big)\cdot(\Delta X_t,\Delta u_t)\\
			&\qquad\le -2C_f|\Delta u_t|^2,\quad t\in[\tau,T];\\
			&\big(f_{1x}(t,{X}_t^2,\hat{m}_t)-f_{1x}(t,{X}_t^1,\hat{m}_t)\big)\Delta{X}_t\geq 0,\quad t\in[\tau,T];\\
			&\Delta p_T\Delta X_T=\big(g(X_T^2,\hat{m}_T)-g(X_T^1,\hat{m}_T)\big)\Delta X_T\geq 0.
		\end{split}
	\end{equation}
	From \eqref{C1_1} and \eqref{C1_2}, we deduce that
	\begin{equation*}
		[v(x_2,\eta)-v(x_1,\eta)](x_2-x_1)=\Delta p_\tau\Delta x\geq 0.
	\end{equation*}
	
	\subsection{Proof of \eqref{exi_lem2_v2}.}
	Let $x_1,x_2\in\br$ and $\eta_1,\eta_2\in \lr^2_{\f_\tau}$. We denote by $\{(\hat{X}_t^i,\hat{p}_t^i,\hat{q}_t^i,\hat{\tilde{q}}_t^i),\tau\le t\le T\}$ the solutions of FBSDE \eqref{lem5.5_} with initial conditions $\hat{X}^i_\tau=\eta_i$, and set $\hat{m}_t^i=\lr(\hat{X}^i_t|\tilde{\f}_t)$ for $t\in[\tau,T]$ and $i=1,2$. We then denote by $\{(X_t^i,p_t^i,q_t^i,\tilde{q}_t^i),\tau\le t\le T\}$ the solution of FBSDE \eqref{exi_lem2_fbsde} corresponding to $\{\hat{m}^i_t,\tau\le t\le T\}$ with initial conditions $X_\tau^i=x_i$ for $i=1,2$. We set $\Delta X_t=X_t^2-X_t^1$. The differences $\Delta p_t,\Delta q_t,\Delta \tilde{q}_t,\Delta \hat{X}_t,\Delta \hat{p}_t,\Delta \hat{q}_t,\Delta \hat{\tilde{q}}_t,\Delta x,\Delta \eta$ are defined in a similar way. Then $\{(\Delta X_t^i,\Delta p_t^i,\Delta q_t^i,\Delta \tilde{q}_t^i),\tau\le t\le T\}$ satisfy the following FBSDE
	\begin{equation*}
		\left\{
		\begin{aligned}
			&d\Delta X_t=[(b_0(t,\hat{m}^2_t)-b_0(t,\hat{m}^1_t))+b_1(t) \Delta X_t+b_2(t) \Delta u_t]dt\\
			&\quad\quad\quad+[(\sigma_0(t,\hat{m}^2_t)-\sigma_0(t,\hat{m}^1_t))+\sigma_1(t) \Delta X_t+\sigma_2(t)\Delta u_t]dW_t\\
			&\quad\quad\quad+[(\tilde{\sigma}_0(t,\hat{m}^2_t)-\tilde{\sigma}_0(t,\hat{m}^1_t))+\tilde{\sigma}_1(t) \Delta X_t+\tilde{\sigma}_2(t)\Delta u_t]d\tilde{W}_t,\quad t\in(\tau,T];\\
			&d\Delta p_t=-\big[b_1(t)\Delta p_t+\sigma_1(t)\Delta q_t+\tilde{\sigma}_1(t)\Delta \tilde{q}_t+(f_{0x}(t,X^2_t,u^2_t)-f_{0x}(t,X^1_t,u^1_t))\\
			&\quad\quad\quad+(f_{1x}(t,X^2_t,\hat{m}^2_t)-f_{1x}(t,X^1_t,\hat{m}^1_t))\big]dt+\Delta q_tdW_t+\Delta\tilde{q}_td\tilde{W}_t,\quad t\in[\tau,T);\\
			&\Delta X_\tau=\Delta x, \quad \Delta p_T=g_x(X^2_T,\hat{m}^2_T)-g_x(X^1_T,\hat{m}^1_T),
		\end{aligned}
		\right.
	\end{equation*}
	with the condition
	\begin{equation}\label{exi_lem2_0.5}
		b_2(t)\Delta p_t+\sigma_2(t)\Delta{q}_t+\tilde{\sigma}_2(t)\Delta{\tilde{q}}_t+(f_{0u}(t,X^2_t,u^2_t)-f_{0u}(t,X^1_t,u^1_t))=0,\quad t\in[\tau,T].
	\end{equation}
	Applying Itô's lemma to $\Delta p_t\Delta X_t$ from $\tau$ to $T$ and taking expectation conditional on $\tilde{\f}_\tau$ (denote by $\ee[\cdot]$), from \eqref{exi_lem2_0.5}, the convexity assumption (H4) and the Lipschitz assumption (H5), we have
	\begin{equation}\label{exi_lem2_1'}
		\begin{split}
			&\ee[\Delta p_T\Delta X_T]-\ee[\Delta p_\tau\Delta x]\\
			&\le -2C_f \ee\big[\int_{\tau}^T |\Delta u_t|^2 dt\big]+L\ee\big[\int_{\tau}^T W_2(\hat{m}^1_t,\hat{m}_t^2) (|\Delta X_t|+|\Delta p_t|+|\Delta q_t|+|\Delta\tilde{q}_t|)  dt\big].
		\end{split}
	\end{equation}
	From assumptions (H4) and (H5), we have 
	\begin{align}
		&\ee[\Delta p_T\Delta X_T]\notag\\
		&=\ee\big[(g_x(X_T^2,\hat{m}^2_T)-g_x(X_T^1,\hat{m}^2_T))\Delta X_T\big]+\ee\big[(g_x(X_T^1,\hat{m}^2_T)-g_x(X_T^1,\hat{m}^1_T))\Delta X_T\big]\notag\\
		&\geq -L\ee\big[W_2(\hat{m}^1_T,\hat{m}^2_T)|\Delta X_T|\big].\label{exi_lem2_2'}
	\end{align}
	Plugging \eqref{exi_lem2_2'} into \eqref{exi_lem2_1'} and using the average inequality, we have for any $\epsilon>0$,
	\begin{equation}\label{exi_lem2_4}
		\begin{split}
			\ee\big[\int_\tau^T|\Delta u_t|^2 dt\big]\le& \frac{C(L,T)}{\epsilon C_f^2}\big(|\Delta x|^2+\ee[\sup_{\tau\le t\le T}|\Delta \hat{X}_t|^2]\big)\\
			&+\epsilon\ee\big[\sup_{\tau\le t\le T}|\Delta X_t|^2+|\Delta p_\tau|^2+\int_{\tau}^T|\Delta p_t|^2+|\Delta q_t|^2+|\Delta\tilde{q}_t|^2dt\big],
		\end{split}
	\end{equation}
	where we have used the estimates
	\begin{equation*}
		\ee[W_2(\hat{m}_t^1,\hat{m}_t^2)^2]\le \ee[\e[|\Delta\hat{X}_t|^2|\tilde{\f}_t]]=\ee[|\Delta\hat{X}_t|^2],\quad t\in[\tau,T].
	\end{equation*}
	By standard estimates for SDEs and BSDEs, we have
	
	\begin{align}
		&\ee[\sup_{\tau\le t\le T}|\Delta X_t|^2]\notag\\
		&\qquad\le C(L,T)\big(|\Delta x|^2+\ee[\sup_{\tau\le t\le T}|\Delta \hat{X}_t|^2]+\ee\int_\tau^T|\Delta u_t|^2dt\big);\label{exi_lem2_5}\\
		&\ee\big[\sup_{\tau\le t\le T}|\Delta{p}_t|^2+\int_{\tau}^T|\Delta{q}_t|^2+|\Delta\tilde{q}_t|^2dt\big] \notag\\
		&\qquad\le C(L,T)\ee\big[\sup_{\tau\le t\le T}|(\Delta\hat{X}_t,\Delta{X}_t)|^2+\int_{\tau}^T|\Delta{u}_t|^2dt\big],\label{exi_lem2_1}
	\end{align}
	where $C(L,T)$ stands for some positive constant depending only on $L$ and $T$. 
	From \eqref{exi_lem2_5}, \eqref{exi_lem2_1} and \eqref{exi_lem2_4}, we deduce that when $\epsilon$ is small enough,
	\begin{equation}\label{exi_lem2_6}
		\ee\Big[\int_\tau^T|\Delta u_t|^2 dt\Big]\le C(L,T)(1+\frac{1}{C_f^2})\big(|\Delta x|^2+\ee[\sup_{\tau\le t\le T}|\Delta \hat{X}_t|^2]\big).
	\end{equation}
	Now we plug \eqref{exi_lem2_5} and \eqref{exi_lem2_6} into \eqref{exi_lem2_1} to get
	\begin{equation}\label{exi_lem2_a}
		\ee[|\Delta p_\tau|^2]\le C(L,T)(1+\frac{1}{C_f^2})\big(|\Delta x|^2+\ee[\sup_{\tau\le t\le T}|\Delta \hat{X}_t|^2]\big).
	\end{equation}
	
	Next, we aim to give the estimate of $\ee[\sup_{\tau\le t\le T}|\Delta \hat{X}_t|^2]$. Note that $\{(\Delta \hat{X}_t^i,\Delta \hat{p}_t^i,\Delta \hat{q}_t^i,\Delta \hat{\tilde{q}}_t^i),\tau\le t\le T\}$ satisfy the following FBSDE
	\begin{equation*}
		\left\{
		\begin{aligned}
			&d\Delta \hat{X}_t=[(b_0(t,\hat{m}^2_t)-b_0(t,\hat{m}^1_t))+b_1(t) \Delta \hat{X}_t+b_2(t) \Delta \hat{u}_t]dt\\
			&\quad\quad\quad+[(\sigma_0(t,\hat{m}^2_t)-\sigma_0(t,\hat{m}^1_t))+\sigma_1(t) \Delta \hat{X}_t+\sigma_2(t)\Delta \hat{u}_t]dW_t\\
			&\quad\quad\quad+[(\tilde{\sigma}_0(t,\hat{m}^2_t)-\tilde{\sigma}_0(t,\hat{m}^1_t))+\tilde{\sigma}_1(t) \Delta \hat{X}_t+\tilde{\sigma}_2(t)\Delta \hat{u}_t]d\tilde{W}_t,\quad t\in(\tau,T];\\
			&d\Delta \hat{p}_t=-\big[b_1(t)\Delta \hat{p}_t+\sigma_1(t)\Delta \hat{q}_t+\tilde{\sigma}_1(t)\Delta \hat{\tilde{q}}_t+(f_{0x}(t,\hat{X}^2_t,\hat{u}^2_t)-f_{0x}(t,\hat{X}^1_t,\hat{u}^1_t))\\
			&\quad\quad\quad+(f_{1x}(t,\hat{X}^2_t,\hat{m}^2_t)-f_{1x}(t,\hat{X}^1_t,\hat{m}^1_t))\big]dt+\Delta \hat{q}_tdW_t+\Delta\hat{\tilde{q}}_td\tilde{W}_t,\quad t\in[\tau,T);\\
			&\Delta \hat{X}_\tau=\Delta x, \quad \Delta p_T=g_x(\hat{X}^2_T,\hat{m}^2_T)-g_x(\hat{X}^1_T,\hat{m}^1_T),
		\end{aligned}
		\right.
	\end{equation*}
	with the condition
	\begin{equation}\label{exi_lem2_0.5'}
		b_2(t)\Delta \hat{p}_t+\sigma_2(t)\Delta\hat{q}_t+\tilde{\sigma}_2(t)\Delta\hat{\tilde{q}}_t+(f_{0u}(t,\hat{X}^2_t,\hat{u}^2_t)-f_{0u}(t,\hat{X}^1_t,\hat{u}^1_t))=0,\quad t\in[\tau,T].
	\end{equation}
	As above, by standard estimates for SDEs and BSDEs, there exist two constants $C_1>0$ and $C_2>0$ depending only on $(L,T)$, such that
	\begin{align}
		&\ee[\sup_{\tau\le t\le T}|\Delta \hat{X}_t|^2]\le C(L,T)\ee[|\Delta \eta|^2]+C_1\ee\big[\int_\tau^T|\Delta \hat{u}_t|^2dt\big];\label{exi_lem2_7}\\
		&\ee\big[\sup_{\tau\le t\le T}|\Delta\hat{p}_t|^2+\int_{\tau}^T|\Delta\hat{q}_t|^2+|\Delta\hat{\tilde{q}}_t|^2dt\big]  \le C_2\ee\big[\sup_{\tau\le t\le T}|\Delta\hat{X}_t|^2+\int_{\tau}^T|\Delta\hat{u}_t|^2dt\big].\label{exi_lem2_7.1}
	\end{align}
	From the weak monotonicity condition (H6), we have
	\begin{equation}\label{weak_monotonicity}
		\begin{split}
			&\ee[\Delta \hat{p}_T\Delta \hat{X}_T]=\ee[(g_x(\hat{X}^2_T,\hat{m}^2_T)-g_x(\hat{X}^1_T,\hat{m}^1_T))\Delta \hat{X}_T]\\
			&\qquad=\e[\e[(g_x(\hat{X}^2_T,\hat{m}^2_T)-g_x(\hat{X}^1_T,\hat{m}^1_T))\Delta \hat{X}_T|\tilde{\f}_T]|\tilde{\f}_{\tau}]\geq 0;\\
			&\ee[(f_{1x}(t,\hat{X}_t^2,\hat{m}_t^2)-f_{1x}(t,\hat{X}_t^1,\hat{m}_t^1))\Delta\hat{X}_t]\\
			&\qquad=\e[\e[(f_{1x}(t,\hat{X}_t^2,\hat{m}_t^2)-f_{1x}(t,\hat{X}_t^1,\hat{m}_t^1))\Delta\hat{X}_t|\tilde{\f}_t]|\tilde{\f}_{\tau}]\geq 0,\quad t\in[\tau,T].
		\end{split}
	\end{equation}
	Applying Itô's lemma to $\Delta\hat{p}_t \Delta\hat{X}_t$ from $\tau$ to $T$ and using the average inequality, from \eqref{exi_lem2_0.5'}, \eqref{weak_monotonicity}, the convexity condition (H4) and the Lipschitz condition (H5), we have
	\begin{equation*}
		\begin{split}
			&-\ee[\Delta \hat{p}_\tau\Delta \eta]\\
			&\le\ee\big[\int_{\tau}^T(b_0(t,\hat{m}^2_t)-b_0(t,\hat{m}^1_t))\Delta \hat{p}_t+(\sigma_0(t,\hat{m}^2_t)-\sigma_0(t,\hat{m}^1_t))\Delta \hat{q}_t\\
			&\qquad\qquad +(\tilde{\sigma}_0(t,\hat{m}^2_t)-\tilde{\sigma}_0(t,\hat{m}^1_t))\Delta \hat{\tilde{q}}_t + (b_2(t)\Delta \hat{p}_t+\sigma_2(t)\hat{q}_t+\tilde{\sigma}_2(t)\hat{\tilde{q}}_t)\Delta \hat{u}_t \\
			&\qquad\qquad-\big(f_{0x}(t,\hat{X}_t^2,\hat{u}_t^2)-f_{0x}(t,\hat{X}_t^1,\hat{u}_t^1)\big)\Delta\hat{X}_tdt\big]\\
			&\le-2C_f \ee\big[\int_{\tau}^T |\Delta \hat{u}_t|^2 dt\big]+L_m\ee\big[\int_{\tau}^T (\tilde{\e}_{t}[|\Delta \hat{X}_t|^2])^{\frac{1}{2}}(|\Delta \hat{p}_t|+|\Delta \hat{q}_t|+|\Delta\hat{\tilde{q}}_t|) dt\big]\\
			&\le -2C_f \ee\big[\int_{\tau}^T |\Delta \hat{u}_t|^2 dt\big]+\frac{TL_m}{2}\ee[\sup_{\tau\le t\le T}|\Delta \hat{X}_t|^2]\\
			&\qquad+\frac{3L_m}{2}\ee\big[\int_{\tau}^T (|\Delta \hat{p}_t|^2+|\Delta \hat{q}_t|^2+|\Delta\hat{\tilde{q}}_t|^2) dt\big].
		\end{split}
	\end{equation*}
	By applying average inequality for $\Delta\hat{p}_\tau\Delta \eta$, we have for any $\epsilon>0$,
	\begin{equation}\label{exi_lem2_7.3}
		\begin{split}
			\ee\Big[\int_\tau^T|\Delta\hat{u}_t|^2 dt\Big] \le& \frac{1}{16C_f^2\epsilon}\ee[|\Delta\eta|^2]+\frac{TL_m}{4C_f}\ee[\sup_{\tau\le t\le T}|\Delta \hat{X}_t|^2]\\
			&+(\frac{3L_m}{4C_f}+\epsilon)(T+1)\ee\big[\sup_{\tau\le t\le T}|\Delta\hat{p}_t|^2+\int_{\tau}^T|\Delta\hat{q}_t|^2+|\Delta\hat{\tilde{q}}_t|^2dt\big].
		\end{split}
	\end{equation}
	Plugging \eqref{exi_lem2_7} and \eqref{exi_lem2_7.1} into \eqref{exi_lem2_7.3} and recall that $L_m\le L$, we have
	\begin{equation*}
		\begin{split}
			&\ee\big[\int_\tau^T|\Delta\hat{u}_t|^2 dt\big]\\
			&\le C(L,T)(1+\frac{1}{\epsilon C_f^2}+\frac{1}{C_f}+\epsilon)\ee[|\Delta\eta|^2]\\
			&\qquad+(\frac{TC_1+3(T+1)(C_1+1)C_2}{4C_f}L_m+(T+1)(C_1+1)C_2\epsilon)\ee\big[\int_{\tau}^T|\Delta\hat{u}_t|^2dt\big].
		\end{split}
	\end{equation*}
	The constant
	\begin{equation*}
		\delta=2(TC_1+3(T+1)(C_1+1)C_2)^{-1}
	\end{equation*}
	depends only on $(L,T)$. If $L_mC_f^{-1}\le\delta$, we choose
	\begin{equation*}
		\epsilon=\frac{1}{4(T+1)(C_1+1)C_2}.
	\end{equation*}
	Then, we have
	\begin{equation}\label{exi_lem2_8}
		\begin{split}
			\ee\Big[\int_\tau^T|\Delta\hat{u}_t|^2 dt\Big]\le C(L,T)(1+\frac{1}{C_f}+\frac{1}{C_f^2})\ee[|\Delta\eta|^2].
		\end{split}
	\end{equation}
	Plugging \eqref{exi_lem2_8} into \eqref{exi_lem2_7}, we have
	\begin{equation}\label{exi_lem2_b}
		\ee[\sup_{\tau\le t\le T}|\Delta \hat{X}_t|^2]\le C(L,T)(1+\frac{1}{C_f}+\frac{1}{C_f^2})\ee[|\Delta \eta|^2].
	\end{equation}
	
	Now we plug \eqref{exi_lem2_b} into \eqref{exi_lem2_a}. If $L_mC_f^{-1}\le\delta$, we have 
	\begin{equation*}
		\ee[|\Delta p_\tau|^2]\le C(L,T)(1+\frac{1}{C_f})^4\big(|\Delta x|^2+\ee[|\Delta \eta|^2]\big),
	\end{equation*}
	or equivalently,
	\begin{equation*}
		|v(x_2,\eta_2)-v(x_1,\eta_1)|^2\le C(L,T)(1+\frac{1}{C_f})^4\big(|x_2-x_1|^2+\e[|\eta_2-\eta_1|^2|\tilde{\f}_\tau]\big)
	\end{equation*}
	as desired.
	\endproof
	%%%%%%%%%%%%%%%%%%%%%%%%%%%%%%%%%%%%%%%%%%%%%%%%%%%%%%%%%%%%%%%%%%%%

\end{document}